\let\ORIlabel\label
\let\ORIrefstepcounter\refstepcounter
\AddToHook{package/hyperref/before}{%
  \let\label\ORIlabel
  \let\refstepcounter\ORIrefstepcounter
}
 
\documentclass[hidelinks,onefignum,onetabnum]{siamart220329}

\usepackage{hyphenat}
\usepackage{relsize}
\usepackage{lipsum}
\usepackage{amsmath,amsfonts,amssymb}
\usepackage{graphicx,subfig}
\usepackage{epstopdf}
\usepackage{algorithmic}
\usepackage{bm}
\usepackage{stmaryrd}
\usepackage[export]{adjustbox}
\usepackage{booktabs}
\ifpdf
  \DeclareGraphicsExtensions{.eps,.pdf,.png,.jpg}
\else
  \DeclareGraphicsExtensions{.eps}
\fi

\newsiamremark{assumption}{Assumption}
\newsiamremark{remark}{Remark}
\newsiamremark{hypothesis}{Hypothesis}
\crefname{hypothesis}{Hypothesis}{Hypotheses}
\newsiamthm{claim}{Claim}

\headers{NOSAS for IPDG}{M. S. Fabien, S. Liu, M. Sarkis}

\title{A nonoverlapping spectral additive Schwarz method for 
interior penalty discontinuous Galerkin discretizations 
of Anisotropic Elliptic Problems
\thanks{Submitted to the editors DATE.
}
}

\author{Maurice S. Fabien\thanks{
Center for Computational Science and Engineering, Schwarzman College of Computing,
\\
 Massachusetts Institute of Technology,
            {77 Massachusetts Avenue}, 
            {Cambridge},
            {MA},
            {02139}, 
            {USA}
            (\email{mfabien@mit.edu})
\\
Department of Mathematics, University of Wisconsin-Madison
\\
            { 500 Lincoln Drive}, 
            {Madison},
            {WI},
            {53706}, 
            {USA} 
,}
  \and 
  Sijing Liu\thanks{Department of Mathematics, State University of New York Polytechnic Institute, 100 Seymour Rd, Utica, NY 13502, USA (\email{lius7@sunypoly.edu})}.
  \and
  Marcus Sarkis\thanks{Department of Mathematical Sciences, Worcester Polytechnic Institute, 100 Institute Rd, Worcester, MA 01609, USA (\email{msarkis@wpi.edu})}.
}

\usepackage{amsopn}

\ifpdf
\hypersetup{
  pdftitle={NOSAS for IPDG},
  pdfauthor={M. S. Fabien, S. Liu, M. Sarkis},
}
\fi

\newcommand{\K}{\bm{K}}
\newcommand{\R}{\bm{R}}

\begin{document}

\maketitle
\begin{abstract}
 We design and analyze a nonoverlapping additive Schwarz preconditioner for interior penalty discontinuous Galerkin (IPDG) discretizations of anisotropic elliptic problems. The preconditioned method coupled with a Krylov subspace iteration is shown to be independent of the highly discontinuous (and anisotropic) jump coefficients as well as the subdomain size. To increase efficacy, various auxiliary spaces are considered to reduce the size of the coarse grid operator. We demonstrate how to modify the additive Schwarz preconditioner such that it is applicable to the nonsymmetric IPDG schemes. Several numerical experiments verify the theory and validate the robustness of the preconditioner.
\end{abstract}



\section{Introduction}
	Finite element discretizations of elliptic partial differential equations (PDEs) generally give rise to large sparse linear systems of equations. To obtain approximations for large scale problems in reasonable time frames, fast solvers are mandatory. Among the most efficient linear solvers for elliptic PDEs known to date are the class of multilevel solvers. Broadly speaking, these multilevel solvers are either of domain decomposition (DD) or multigrid type \cite{bramble2019multigrid,toselli2004domain}. In this paper, we focus on a two-level additive Schwarz preconditioner.
	
	The continuous Galerkin (CG) method is arguably the most popular or best studied finite element discretization. As such, extensive research has been done to design and analyze efficient domain decomposition solvers for CG discretizations \cite{brenner2007mathematical,toselli2004domain}. However, in recent years, the discontinuous Galerkin method (DG) has garnered interest from the research community thanks to the flexibility of the nonconforming discretization and completely discontinuous approximation. The reader is referred to \cite{riviere2008discontinuous} for additional details about DG methods. The earliest work on DD methods for DG was conducted in \cite{feng2001two}, where the primary focus was on elliptic problems with smoothly varying coefficients. DG methods give rise to significantly larger linear systems compared to CG, and applications with highly varying discontinuous coefficients can further exacerbate conditioning.
	
	For many applications, it is desirable for preconditioners to have convergence rates independent of mesh size as well as material parameters. As an example, for Darcy flow, it is not unusual for permeability to vary in a discontinuous manner several orders of magnitude \cite{helmig1997multiphase}. This gives rise to very poorly conditioned linear systems after discretization, which require scalable solvers. In the case of highly varying (discontinuous) coefficients, several robust DD solvers exist \cite{galvis2018overlapping,toselli2004domain}. Traditionally, the most effective methods are usually overlapping (with generous overlap). However, overlapping methods reduce parallelism and appropriate overlaps (partitions of unity) increase computational cost, especially in 3D.  Robust DD methods for highly varying coefficients based on nonoverlapping partitions have also been successful, for instance balancing domain decomposition by constraints (BDDC) and finite element tearing and interconnect (FETI) \cite{dryja2008balancing,dryja2015deluxe,klawonn2000feti}. BDDC and FETI are supported by strong theoretical evidence, and have been demonstrated to work well in practice, but are more intricate and complicated than additive average Schwarz type methods \cite{bjorstad1997additive}.
 
	These nonoverlapping DD methods (often referred to as substructuring \cite{toselli2004domain}) partition the domain into nonoverlapping regions, which is typically more natural for mesh and mesh partitioning software \cite{geuzaine2009gmsh,karypis1997metis}. To that end, various nonoverlapping DD solvers for DG methods have been examined, for instance, \cite{dryja2007bddc,dryja2008balancing,dryja2010n,dryja2013feti,dryja2015analysis,dryja2016massively,dryja2010additive}. Most of these works assume that the coefficients are constants, piecewise constants with respect to some partition of the domain, or have a ``quasi-monotonicity'' property \cite{bjorstad1997additive}. It is generally known that the choice of coarse space is very important to ensure a robust and scalable solver; however, finding these coarse spaces can be challenging.  
	
	A significant breakthrough for DD methods with highly discontinuous coefficients was made when effective coarse spaces were found \cite{dolean2015introduction}. In particular, the so-called adaptive spectral coarse spaces were deemed to be extremely robust \cite{dolean2015introduction,liu2020two,spillane2014abstract,yu2021additive,yu2024family}. These spaces are formed by considering specific generalized eigenvalue problems on each subdomain defined by Dirichlet-to-Neumann mapping, or other suitable bilinear forms \cite{bastian2022multilevel,dolean2012analysis,efendiev2012robust,eikeland2017adaptively, galvis2010domain,graham2007domain,nataf2010two,nataf2011coarse,scheichl2007additive,spillane2014abstract,willems2014robust}.  The idea for considering nonstandard coarse spaces have been investigated earlier as well \cite{chartier2003spectral,sarkis2003partition}.
 
 Despite the success of the spectral coarse space, to date, few works study their application to DG methods \cite{bastian2022multilevel,eikeland2017adaptively,liu2020two,yu2021additive}. In \cite{yu2021additive}, a spectral additive Schwarz method was studied for a piecewise constant hybridizable discontinuous Galerkin scheme \cite{cockburn2016static}. Here the coarse space is easier to construct due to its algebraic formulation, but only the lowest order scheme is considered.
  The works in \cite{eikeland2017adaptively,liu2020two} pose a specific Dirichlet-to-Neumann mapping to define the generalized eigenvalue problems on each subdomain. The choice of an effective Dirichlet-to-Neumann can be delicate.  The primary focus of \cite{bastian2022multilevel} was to examine multilevel (more than two) DD methods, although they consider higher order DG schemes, the methods therein require partitions of unity.
	 
	In this paper, we propose a nonoverlapping spectral additive Schwarz method for interior penalty DG schemes. To build the coarse space, only information from the original discretization is required, rendering the method more algebraic (no Dirichlet-to-Neumann mapping is needed, similar to \cite{yu2021additive}). As the DG method gives rise to a proliferation of unknowns, alternative coarse spaces as well as strategies to address large coarse spaces are discussed. Moreover, to the best of our knowledge, we also introduce the first study of spectral coarse spaces for nonsymmetric interior penalty discontinuous Galerkin schemes. Also, we test our method on anisotropic problems.
 
	The paper is organized as follows. \Cref{sec:prelim} defined and clarifies the preliminary information required to describe the DG method, and domain decomposition. In \Cref{sec:matrix_coarse}, the algebraic formulation and explicit construction of the coarse space is presented. This section also discusses various alternative coarse space options as well as modifications required to apply the preconditioner to nonsymmetric problems. Theoretical properties of the preconditioner are rigorously established in \Cref{sec:theory}; of particular interest is the independence of the condition number of the preconditioned Schwarz operator with respect to the subdomain size and highly discontinuous coefficients. Verification and validation of the preconditioner is done in \Cref{sec:num_ex}. Several numerical experiments are conducted to test the robustness and efficiency of the solver. Finally, the findings of the paper are summarized in the conclusion.

\section{Preliminaries and notations} \label{sec:prelim}
We consider a bounded polygonal domain $\Omega \subset \mathbb{R}^d$ for $d=2,3$. The boundary of the domain $\partial\Omega$ is grouped into two disjoint sets, $\partial\Omega = \Omega_D \cup \Omega_N$ (Dirichlet and Neumann boundaries, respectively). Let $\bf n$ be the unit normal vector to the boundary exterior to $\Omega$. We assume that $f\in L^2(\Omega)$, $g_D\in H^{1/2}(\partial\Omega_D)$, $g_N\in H^{1/2}(\partial\Omega_N)$, and $\K$ is a matrix-valued function ${\K}= (k_{ij})_{1\le i,j\le d}$ that is symmetric positive definite and uniformly bounded above and below. That is, for all $x\in \mathbb{R}^d$, there exists constants $K_0$ and $K_1$ such that $K_0 x \cdot x \le \K x \cdot x \le  K_1 x \cdot x$. In this paper, we are particularly interested in the case where $\K$, when restricted to an element, is a constant matrix, but the magnitude of the values of the matrix can jump rapidly. For instance, if $E_1,E_2\in \mathcal{T}_h$ share a facet, $\max_{1\le i,j,\le d} (\K^{E_1})_{ij} \ll \max_{1\le i,j,\le d} (\K^{E_2})_{ij}.$ Although we allow $\K$ to be highly heterogeneous, we can always scale the problem so that $\max_{1\le i,j,\le d} (\K^{E})_{ij} \ge 1$ for all $E\in \mathcal{T}_h$.

The model problem is the Poisson equation:
\begin{subequations}\label{eq:model}
\begin{align}
-\nabla \cdot ( \K \nabla u ) + \alpha u & = f, &&\textrm{ in } \Omega
\\
u &= g_D, &&\textrm{ on } \partial\Omega_D,
\\
\K \nabla u \cdot {\bm n} &=g_N, &&\textrm{ on } \partial\Omega_N.
\end{align} 
\end{subequations}
 
Let $\mathcal{T}_h$ be a shape-regular triangulation (see \cite{toselli2004domain}) of the domain $\Omega$. In this work, we assume that the mesh comprises of simplicial elements (triangles in 2D, and tetrahedra in 3D). The mesh skeleton is denoted by $\mathcal{E}_h$ contains all the facets (edges in 2D, faces in 3D) of the mesh. As both Dirichlet and Neumann boundary conditions are possible, the set $\mathcal{E}_h$ is further partitioned as $\mathcal{E}_h =\mathcal{E}_h^\circ \cup \mathcal{E}_h^D \cup \mathcal{E}_h^N $, where $\mathcal{E}_h^D\subset  \partial\Omega_D$, $\mathcal{E}_h^N\subset  \partial\Omega_N$, and $\mathcal{E}_h^\circ$ are the interior facets.

\subsection{The DG scheme}
Let $k$ be a positive integer. The discontinuous finite element space is defined as
\[
V_h:= V_h(\Omega) = \{ v\in L^2(\Omega): \forall E \in \mathcal{T}_h,~~v|_{T} \in \mathcal{P}_k(E)\},
\]
where $\mathcal{P}_k(E)$ is the space of polynomials of total degree less than or equal to $k$.

These functions in $V_h$ are double-valued on interior facets $e\in \mathcal{E}_h^\circ$. Let $v^{\pm}$ on an interior face $e\in \mathcal{E}_h^\circ$ be the restriction of $v$ to $E^{\pm}$ ($e\in E^-\cap E^+$). The jump and weighted average for $v\in V_h$ on $e\in \mathcal{E}_h^\circ$ are given by
\[
[v] = v^- - v^+,
\quad
\{v\}_{\omega} = \omega^- v^- + \omega^+ v^+,
\]
where $\omega^{\pm}\ge0$ and they are defined as follows:
\begin{equation}
    \omega^-=\frac{{\bf n}^T\K^{E^+}{\bf n}}{{\bf n}^T\K^{E^+}{\bf n}+{\bf n}^T\K^{E^-}{\bf n}},\quad \omega^+=\frac{{\bf n}^T\K^{E^-}{\bf n}}{{\bf n}^T\K^{E^+}{\bf n}+{\bf n}^T\K^{E^-}{\bf n}}.
\end{equation}
Here $\K^{E^-}$ is the permeability of $E^-$ and $\K^{E^+}$ is the permeability of $E^+$. If $e\in\mathcal{E}_h^D$, the jump and the average for $v\in V_h$ is defined as
\begin{equation}
    [v]=v^+,\quad \{v\}_w=v^+.
\end{equation}
 
 The interior penalty discontinuous Galerkin (IPDG) bilinear form is denote by $a(\cdot,\cdot):V_h\to \mathbb{R}$, and is defined as
 \begin{equation} \label{eq_dg_disc}     
\begin{aligned} 
a(v,w) &:=  \sum_{E\in \mathcal{T}_h}\int_E \K \nabla v \cdot \nabla w-\sum_{e\in\mathcal{E}_h^\circ\cup \mathcal{E}_h^D}\int_e   \{ \K \nabla v \cdot {\bf n} \}_{\omega} [w]
 \\&+
 \epsilon 
 \sum_{e\in\mathcal{E}_h^\circ\cup \mathcal{E}_h^D}\int_e  \{ \K \nabla w \cdot {\bf n} \}_{\omega}  [v]
 +
 \sum_{e\in\mathcal{E}_h^\circ\cup \mathcal{E}_h^D}\int_e \sigma_e [v][w].
\end{aligned}
 \end{equation}
The scalar $\epsilon$ is defined to be the symmetrization parameter. When $\epsilon=-1$, we obtain the symmetric interior penalty method (SIPG), $\epsilon= 1$, gives rise to the nonsymmetric interior penalty method (NIPG), and $\epsilon = 0$, gives rise to the incomplete interior penalty method (IIPG). We note that NIPG and IIPG are nonsymmetric, but are commonly used in many applications (for a nonexhaustive list:  \cite{bastian2014fully,bastian2012algebraic,dolejvsi2008semi, fabien2022numerical,li2015high,liu2009modeling,sun2005symmetric}).

Let ${\bm n}$ be the outward normal associated with facet $e\in \mathcal{E}_h^\circ$. For $e \in E^+\cap E^-$, $E^{\pm}\in \mathcal{T}_h$, we let $\delta_{K_n}^{\pm} = {\bm n}^T {\K^{E^{\pm}}} {\bm n}$ and for $e\in \mathcal{E}_h^D \cup \mathcal{E}_h^N$, set $\delta_{K_n}  = {\bm n}^T \K^E {\bm n}$. The penalty parameter $\sigma_e$ is given by
\begin{equation} 
\sigma_e = 
\begin{cases}
 \sigma \frac{2 k(k+d-1) \delta_{K_n}^- \delta_{K_n}^+ }{\delta_{K_n}^- + \delta_{K_n}^+} \frac{|e|}{\min\{ |E^+|, |E^-|\} }    &  \forall e \in \mathcal{E}_h^\circ
\\
\sigma
k(k+d-1) \delta_{K_n} \frac{|e|}{|E|}
&  \forall e \in \mathcal{E}_h^D \cup \mathcal{E}_h^N
\end{cases}
\label{eq:dg_penalty}
, 
\end{equation}
where $\sigma$ is a nonnegative user-defined parameter. The weighted IPDG schemes with penalty parameters dependent on the anisotropic diffusivity have been deemed important for stability and accuracy purposes \cite{epshteyn2007estimation,ern2009discontinuous,hartmann2008optimal}.
 
For any $v\in V_h$, the associated linear function $\ell(\cdot)$ is given by
\begin{align}
\ell(v) &= \sum_{E\in \mathcal{T}_h}\int_E f v  
+
\sum_{e\in \mathcal{E}_h^D}\int_e   ( \K \nabla v \cdot {\bf n} ) g_D
+
\sum_{e\in \mathcal{E}_h^N}\int_e   v g_N.
 \label{eq_dg_disc3}
\end{align} 
The general IPDG finite element method is as follows: seek $u_h\in V_h$ such that
\begin{align} 
a(u_h,v_h) = \ell(v_h),\quad \forall v_h \in  V_h.
 \label{eq_dg_disc2}
\end{align}
The associated DG norm for $v_h \in V_h$ is given by 
\begin{equation} 
\| v_h\|_{DG}^2 = 
\sum_{E\in \mathcal{T}_h}\int_E \K \nabla v_h \cdot \nabla v_h
+
 \sum_{e\in\mathcal{E}_h^\circ\cup \mathcal{E}_h^D}\int_e \sigma_e [v_h][v_h].
\label{eqn_dg_norm}
\end{equation}
\subsection{Domain and DG space decomposition}
 The additive Schwarz preconditioner requires a decomposition of the domain. We focus on a nonoverlapping method, so that the mesh $\mathcal{T}_h$ is partitioned into $N$ subdomains, denoted by $\Omega_i$.  That is, we decompose $\Omega$ into $N$ nonoverlapping open polygonal subdomains $\Omega_i$ with diameter $\mathcal{O}(H)$ where
 \[
 \overline{\Omega}
 =
 \bigcup_{i=1}^N \overline{\Omega}_i
 ,
 \quad 
 \textrm{ and } 
 \quad 
 \Omega_i \cap \Omega_j = \emptyset,
  \quad 
  i\neq j.
 \]
  Each subdomain is assumed to be the union of shape regular simplicies. The interface of $\Omega_i$ is denoted by $\Gamma_i$, and the global interface $\Gamma$ given by 
  \[
  \Gamma_i := \partial \Omega_i \backslash \partial\Omega,
  \quad
  \textrm{and} 
  \quad
  \Gamma := \bigcup_{i=1}^N \Gamma_i.
  \]
  See \Cref{fig:subdomain_doodle} for a sample depiction. The local finite element space 
  $V_i$,  for $i=1,2,\ldots,N$, is the restriction of $V_h(\Omega)$ to $\Omega_i$ and vanishes on $\Gamma_i$. 

The extension by zero outside of $\Omega_i$ is given by the operator $\R_i^T: V_i \to V_h(\Omega).$ The operator $\R_i :    V_h(\Omega)\to V_i $ (the adjoint of $\R_i^T$) maps a nodal vector on $\Omega$ to a nodal vector inside $\Omega_i$. As the nodal DG method has duplication of unknowns at the mesh vertices, we consider interface unknowns to be those residing on facets belonging to the interface (see \Cref{fig:subdomain_doodle2}). 
\begin{figure}[htb!]
    \centering

    \subfloat[\centering Unstructured partitioning\label{fig:subdomain_doodle}]{
        \includegraphics[scale=0.5,valign=t]{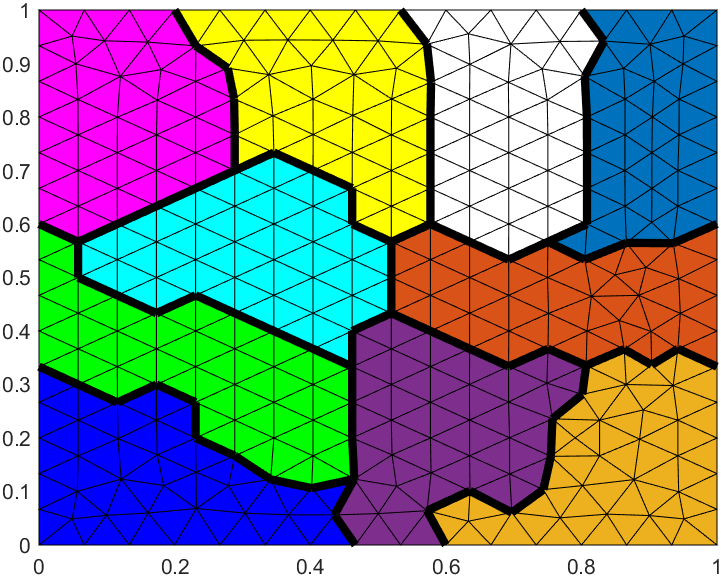}
    }
    \qquad
    \subfloat[\centering Interface unknowns\label{fig:subdomain_doodle2}]{
        \includegraphics[scale=0.9,valign=t]{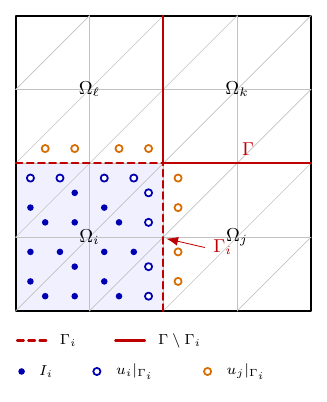}
    }

    \caption{In (a), decomposition of the unit square into 10 nonoverlapping subdomains and global interface (thick solid line). Thin solid lines are element boundaries which lie in the interiors of the subdomain. The domain partitioning does not have to be aligned with varying coefficient $\K$. In (b), the global interface is given by the red line. Illustration of some unknowns that lie on the global interface (circles), and unknowns that are in the subdomain interior (filled circles).}
    \label{fig:example}
\end{figure}

 The coarse space is defined as 
 \[
 V_0 = V_h(\Gamma) := \{v |_e, e\in \mathcal{E}_h \subset \Gamma; \forall v \in V_h(\Omega)\}.
 \]
 We note that $V_h(\Gamma)$ has duplication of unknowns on $\Gamma$.  Then, the DG space $V_h(\Omega)$ can be written as a direct sum:
\[
V_h(\Omega) = \R^T_0 V_0 \oplus\R^T_1 V_1 \oplus \ldots \oplus \R^T_N V_N,
\]
where $\R^T_0:V_0\to V_h(\Omega)$ is defined later.

The DG method \eqref{eq_dg_disc2} gives rise to a linear system of the form
\[
{\bf A} \vec{u} = \vec{b}
\]
in terms of the standard basis functions, where we organized the unknowns in terms of the interface and interior of the subdomains. That is,
\[
\begin{bmatrix} \label{Assembly}
{\bf A}_{\Gamma\Gamma}  & {\bf A}_{\Gamma I} 
\\
{\bf A}_{I\Gamma}  & {\bf A}_{II} 
\end{bmatrix}\begin{bmatrix}
\vec{u}_{\Gamma}
\\
\vec{u}_{I}
\end{bmatrix}  
=
\sum_{i=1}^N \R^{(i)^T}
\begin{bmatrix} 
{\bf A}_{\Gamma\Gamma}^{(i)}  & {\bf A}_{\Gamma I}^{(i)}  
\\
{\bf A}_{I\Gamma}^{(i)}   & {\bf A}_{II}^{(i)}  
\end{bmatrix}
\R^{(i) }
\begin{bmatrix}
\vec{u}^{(i) }_{\Gamma}
\\
\vec{u}^{(i) }_{I}
\end{bmatrix}
=
\sum_{i=1}^N \R^{(i)^T}
\begin{bmatrix}
\vec{b}^{(i) }_{\Gamma}
\\
\vec{b}^{(i) }_{I}
\end{bmatrix},
\]
where it is understood that $\vec{u}_{\Gamma}$ ($\vec{u}_{I}$) is the restriction of $\vec{u}$ on $\Gamma$ ($I=\Omega\backslash\Gamma)$. The terms $\vec{b}^{(i)}_{\Gamma}$ and $\vec{b}^{(i)}_{I}$ represent the restriction of $\vec{b}$, which represents the linear form \eqref{eq_dg_disc3}, to the interface $\Gamma_i$ and to $I_i=\Omega_i\backslash\Gamma_i$, that is the interior degrees of freedom inside $\Omega_i$, respectively. We note that $V_i = V_h(I_i).$

We define several prolongation and restriction operators to map from various topological objects associated with the domain decomposition:
\begin{itemize}
\item  $\R^T_{\Gamma_i}:  V_h(\Gamma_i) \to V_h(\Omega)  $ (extension to $\Omega$ by zero outside $\Gamma_i$),
\item  $\R^T_{i}: V_h(I_i)\to  V_h(\Omega)  $  (extension by zero outside $I_i$),
\item  $\R^T_{\Gamma_i \Gamma}: V_h(\Gamma_i)\to V_h(\Gamma)$ (extension to $\Gamma$ by zero outside $\Gamma_i$),
\item  $\R^T_{I_i I}: V_h(I_i)\to V_h(I)$ (extension to $I$ by zero outside $I_i$),
\item  $\R^{(i)^T} = 
[R^T_{\Gamma_i},R^T_{i}]
=
\begin{bmatrix}
\R^T_{\Gamma_i \Gamma} & 0
\\
0                      & \R^T_{I_i I}
\end{bmatrix}
$
 (prolongation operator from $\Omega_i$ to $\Omega$).
\end{itemize}
We note that $V_h(\Gamma_i):=\{v_i|_{\Gamma_{i}}, v_j|_{\Gamma_{i}}\ \forall v_i\in V_h(\Omega_i), v_j\in V_h(\Omega_j)\ \mbox{and}\  j\in\partial i\}$, where $\partial i$ is the set of indices $j$ such that $\Omega_j$ and $\Omega_i$ share a common interface $\Gamma_{i}$. 

It should be emphasized that the local matrices 
\[
{\bf A}^{(i)}=\begin{bmatrix}
{\bf A}_{\Gamma\Gamma}^{(i)}  & {\bf A}_{\Gamma I}^{(i)}  
\\
{\bf A}_{I\Gamma}^{(i)}   & {\bf A}_{II}^{(i)}  
\end{bmatrix}\]
 are the Neumann matrices corresponding to the restriction of the bilinear form from equation \eqref{eq_dg_disc} to subdomain $\Omega_i$. That is,
\[
a^{(i)}(u_i,v_i) := a(\R^{(i)^T}u_i,\R^{(i)^T}v_i),
\quad \forall u_i,v_i\in V_i,
\quad \text{and}\ 
{\bf A} = \sum_{i=1}^N \R^{(i)^T} {\bf A}^{(i)} \R^{(i) }.
\]
The matrices ${\bf A}^{(i)}$ are sometimes referred to as the unassembled contributions before direct stiffness summation \cite{toselli2004domain} and it represents the following bilinear form
\begin{align*} 
a^{(i)}(v,w) &:=\sum_{E\in \Omega_i}\int_E \K\nabla v \cdot \nabla w 
\\
&+\sum_{ e \textrm{ interior to } \Omega_i}
\bigg( - \int_e \{\K\nabla v \cdot {\bf n}\}_\omega [w] +\epsilon  \{\K\nabla w \cdot {\bf n}\}_\omega [v] + \sigma_e [v][w]\bigg)
\\
&+\sum_{\substack{e \textrm{ on } \partial\Omega_i \cap \Gamma_i,\\
e=E_i\cap E_j}}
\int_{e}
\bigg(
-   \omega_i  (\K^{E_i}\nabla v_i \cdot {\bf n}_{ij})[w^{(i)}] 
+  \epsilon  \omega_i  (\K^{E_i}\nabla w_i \cdot {\bf n}_{ij}) [v^{(i)}] 
\\
&\hspace{3cm}+ 
0.5 \sigma_e [v^{(i)}] [w^{(i)}] 
\bigg),
\end{align*}
where in the last summation term $v_i, w_i\in\Omega_i$, $v^{(i)}, w^{(i)}\in V_h(\Gamma_i)$ and we define
\begin{equation}
    \omega_i=\frac{{\bf n}^T\K^{E_j}{\bf n}}{{\bf n}^T\K^{E_i}{\bf n}+{\bf n}^T\K^{E_j}{\bf n}}.
\end{equation}
The vector $\mathbf{n}_{ij}$ represents the normal vector on $e$ pointing from $E_i$ to $E_j$. Fig.~\ref{fig:example}(b) visualizes a sample domain, its decomposition, and its degree of freedom configuration.

The Schwarz preconditioner is based on the Schur complement. That is, for the problem ${\bf A} \vec{u} = \vec{b}$  we can write
\[
\begin{bmatrix}
{\bf S}  & 0
\\
{\bf A}_{I\Gamma}  & {\bf A}_{II} 
\end{bmatrix}
\begin{bmatrix}
\vec{u}^{(i) }_{\Gamma}
\\
\vec{u}^{(i) }_{I}
\end{bmatrix}
=
\sum_{i=1}^N \R^{(i)^T}
\begin{bmatrix}
{\bf S}^{(i)}  &0 
\\
{\bf A}_{I\Gamma}^{(i)}   & {\bf A}_{II}^{(i)}  
\end{bmatrix}
\R^{(i) }
\begin{bmatrix}
\vec{u}^{(i) }_{\Gamma}
\\
\vec{u}^{(i) }_{I}
\end{bmatrix}
=
\sum_{i=1}^N \R^{(i)^T}
\begin{bmatrix}
\widehat{b}_\Gamma
\\
\vec{b}^{(i) }_{I}
\end{bmatrix},
\]
 where
 \begin{align*}
 \widehat{b}_\Gamma^{(i) } &= \vec{b}^{(i) }_{\Gamma} - {\bf A}_{ \Gamma I}^{(i)}({\bf A}_{II}^{(i)})^{-1}\vec{b}^{(i) }_{I},
 \\
 {\bf S}^{(i)} &= {\bf A}_{ \Gamma \Gamma }^{(i)}-   {\bf A}_{ \Gamma I}^{(i)}({\bf A}_{II}^{(i)})^{-1}{\bf A}_{ I\Gamma  }^{(i)},
 \\
 {\bf S} &= \sum_{i=1}^N \R^{(i)^T}_{\Gamma_i\Gamma} {\bf S}^{(i)}\R^{(i) }_{\Gamma_i\Gamma}
 \\
 \widehat{b}&= \sum_{i=1}^N \R^{(i)^T}_{\Gamma_i\Gamma}  \widehat{b}^{(i) }
 .
 \end{align*}
 The Schur complement results in problem of reduced size (for the unknowns on the interface),
 \[
  {\bf S}  \vec{u}_\Gamma = \widehat{b}.
 \] 
 The $a$-discrete harmonic operator $\mathcal{H}:V_0\to V_h(\Omega)$ is given by
 \[
  \mathcal{H} u_\Gamma = 
  \begin{cases}
  u_\Gamma, &\mbox{ on } \Gamma 
  \\
  -\sum_{i=1}^N  
   \R^T_{I_i I}({\bf A}^{(i)}_{II})^{-1}{\bf A}^{(i)}_{I\Gamma}
   \R_{\Gamma_i \Gamma}
  u_\Gamma, &\mbox{ on }  I
  \end{cases}.
 \]
We can then define the bilinear form for the Schur complement. For all $u_\Gamma,v_\Gamma \in V_0$,
\[
s(u_\Gamma,v_\Gamma) = v_\Gamma^T {\bf S} u_\Gamma = a( \mathcal{H} u_\Gamma , \mathcal{H} v_\Gamma ).
\]
If we selected $\R_0^T =  \mathcal{H}$, then the course problem would coincide with the Schur complement, but this is too costly in practice. Instead, we design $\R_0^T$ so that the coarse grid bilinear form
\begin{equation} 
a_0(u_\Gamma,v_\Gamma):= a(\R_0^Tu_\Gamma,\R_0^Tv_\Gamma)
\label{eq_coarse_grid_bilinear}
\end{equation}
 is spectrally equivalent to $s(u_\Gamma,v_\Gamma)$ independent of the heterogeneity from $\K$; in some subspaces (to be determined). In other words, $\R_0^T$ is a lower-rank approximation to $\mathcal{H}$.

\subsection{Spectral coarse space} \label{sec:coarse_space}
In this section we describe $\R_0^T$ and its construction. The so-called adaptive spectral coarse spaces are utilized \cite{dolean2015introduction,liu2020two,spillane2014abstract,yu2021additive,yu2024family}. These spaces are formed by considering specific generalized eigenvalue problems on each subdomain defined by Dirichlet-to-Neumann mapping \cite{graham2007domain,nataf2010two,nataf2011coarse,scheichl2007additive,spillane2014abstract,willems2014robust}, or other suitable bilinear forms  \cite{bastian2022multilevel,dolean2012analysis,efendiev2012robust,eikeland2017adaptively,galvis2010domain}.

The process of constructing spectral coarse spaces is typically done in two steps. First, in each subdomain, a local generalized eigenvalue problem is solved, where eigenfunctions corresponding to the small eigenvalues by a given threshold eigenvalue are retained. Second, the subdomain contributions of said eigenfunctions are combined via an appropriate partition of unity to arrive at a global coarse space. Such spectral coarse spaces result in very robust preconditioners which are independent of the number of subdomains as well as the heterogeneity of the coefficients $\K$ \cite{dolean2015introduction}.

The following generalized eigenvalue problems are considered on each subdomain:
\[
{\bf S}^{(i)}\xi_j^{(i)}
=
({\bf A}_{\Gamma\Gamma}^{(i)} - {\bf A}_{\Gamma I}^{(i)} 
({\bf A}_{II}^{(i)})^{-1}
{\bf A}_{I\Gamma}^{(i)}) 
\xi_j^{(i)}
=
\lambda_j^{(i)} {\bf B}_{\Gamma\Gamma}^{(i)} \xi_j^{(i)},   
\]
for $i=1,2,\ldots,N$, and $j=1,2,\ldots n_i$, where $n_i$ are the number of degrees of freedom on $\Gamma_i$. We have several choices for the matrix ${\bf B}_{\Gamma\Gamma}^{(i)}$:
\begin{itemize}
    \item \({\bf B}_{\Gamma\Gamma}^{(i)}\) can be \({\bf A}_{\Gamma\Gamma}^{(i)}\).

    \item \({\bf B}_{\Gamma\Gamma}^{(i)}\) can be the block diagonal of
    \({\bf A}_{\Gamma\Gamma}^{(i)}\), denoted by
    \(\widehat{\bf A}_{\Gamma\Gamma}^{(i)}\). Then
    \(\widehat{\bf  A}_{\Gamma\Gamma}\), the assembly of
    \(\widehat{\bf A}_{\Gamma\Gamma}^{(i)}\), is the same as the block diagonal
    of \({\bf A}_{\Gamma\Gamma}\).

    \item \({\bf B}_{\Gamma\Gamma}^{(i)}\) can be the block diagonal of
    \({\bf S}^{(i)}\), denoted by \(\widehat{{\bf S}}_{\Gamma\Gamma}^{(i)}\). Then
    \(\widehat{\bf S}_{\Gamma\Gamma}\), the assembly of
    \(\widehat{\bf S}_{\Gamma\Gamma}^{(i)}\), is the same as the block diagonal
    of \({\bf S}_{\Gamma\Gamma}\).
\end{itemize}
For piecewise linear DG ($k=1$), by a block, we mean that each block is associated with either an edge or a corner of \(\Gamma\). The size of each edge block is the number of nodes on the edge from both sides, excluding the endpoints. For a structured subdomain, the size of each corner block is \(3 \times 3\).

The choice of ${\bf B}_{\Gamma\Gamma}^{(i)}$ can increase or decrease the computational cost of the method. As we will see below, ${\bf B}_{\Gamma\Gamma}^{(i)}$ influences the coarse grid operator, and. Moreover, it turns out that the choice ${\bf B}_{\Gamma\Gamma}^{(i)} = \widehat{{\bf S}}_{\Gamma\Gamma}^{(i)}$ is very robust, but requires assembly of the local Schur complement ${\bf S}^{(i)}$.


Notice that each of these eigenvalue problems are independent from one another. Moreover, since since the ${\bf S}^{(i)}$ are symmetric positive semidefinite and all the three choices of $ {\bf B}_{\Gamma\Gamma}^{(i)}$ are symmetric positive definite, we have $0\le \lambda_j^{(i)}$. Since  ${\bf S}^{(i)}$ is related to minimum energy extension while 
${\bf A}_{\Gamma\Gamma}^{(i)}$ is related to zero trivial extension, we have $\lambda_j^{(i)} \leq 1$, while for the other two choices we will show that $\lambda_j^{(i)} \leq O(1)$. 

The spectral coarse space begins by defining a threshold parameter $\eta = \mathcal{O}(\frac{h}{H})$. We then keep the $k_i$ eigenvalues smaller than $\eta$ in each subdomain $\Omega_i$ (low frequency modes). Associated with each subdomain, we can rewrite the generalized eigenvalue problem as
\[
\begin{bmatrix}
{\bf S}^{(i)} & 0 
\\
 {\bf A}_{I\Gamma}^{(i)}             & {\bf A}_{II}^{(i)}
\end{bmatrix}
\begin{bmatrix}
{\bf Q}^{(i)}
\\
{\bf P}^{(i)}
\end{bmatrix}
=
\begin{bmatrix}
 {\bf B}_{\Gamma\Gamma}^{(i)}& 0 
\\
0& 0
\end{bmatrix}
\begin{bmatrix}  
{\bf Q}^{(i)}{\bf D}^{(i)}
\\
{\bf P}^{(i)}{\bf D}^{(i)}
\end{bmatrix}
,
\]
where ${\bf P}^{(i)} = - ({\bf A}_{II}^{(i)})^{-1} {\bf A}_{I\Gamma}^{(i)} {\bf Q}^{(i)}$, is the best low-rank ($k_i$) subspace of $V_i$, ${\bf Q}^{(i)}  = [\xi_1^{(i)},\ldots,\xi_{k_i}^{(i)}]$ contains the generalized eigenvectors, and ${\bf D}^{(i)}=\textrm{ diag}(1-\lambda_1,\ldots,1-\lambda_{k_i})$.

We can now define the local extension operators $\R_0^{(i)^T}: V_h(\Gamma_i) \to V_h(\Omega_i)$
\[
  \R_0^{(i)^T} u_{\Gamma_i} = 
  \begin{cases}
  u_{\Gamma_i}, &\mbox{ on } \Gamma_i
  \\
  -{\bf P}^{(i)}({\bf Q}^{(i)^T} {\bf B}_{\Gamma\Gamma}^{(i)}{\bf Q}^{(i)})^{-1}
  {\bf Q}^{(i)^T}
  {\bf B}_{\Gamma\Gamma}^{(i)}
  u_{\Gamma_i}, 
  &\mbox{ on }  \Omega_i^\circ
  \end{cases},
\]
and $\R_0^{T}: V_h(\Gamma) \to V_h(\Omega)$
\[
  \R_0^{T} u_\Gamma = 
  \begin{cases}
  u_{\Gamma }, &\mbox{ on } \Gamma 
  \\
  -\sum_{i=1}^N
  \R_{I_i I}^T {\bf P}^{(i)}({\bf Q}^{(i)^T} {\bf B}_{\Gamma\Gamma}^{(i)}{\bf Q}^{(i)})^{-1}
  {\bf Q}^{(i)^T}
  {\bf B}_{\Gamma\Gamma}^{(i)}
  \R_{\Gamma_i \Gamma} u_{\Gamma }, 
  &\mbox{ on }  \Omega^\circ
  \end{cases},
\]
where $\Omega^\circ$ is the interior of $\Omega.$
\section{Matrix formulation of the spectral coarse space} \label{sec:matrix_coarse} 
   The coarse problem defined by the bilinear form $a_0(\cdot,\cdot)$ in equation \eqref{eq_coarse_grid_bilinear} can then be assembled as follows:
   \[
   {\bf A}_0 =
   \sum_{i=1}^N
   \R_{\Gamma_i \Gamma}^T
   \bigg(
   {\bf B}_{\Gamma\Gamma}^{(i)}
   -
{\bf B}_{\Gamma\Gamma}^{(i)}   
{\bf Q}^{(i)}
{\bf D}^{(i)}   
   ({\bf Q}^{(i)^T} {\bf B}_{\Gamma\Gamma}^{(i)}{\bf Q}^{(i)})^{-1}
   {\bf Q}^{(i)^T}
  {\bf A}_{\Gamma\Gamma}^{(i)}
  \bigg)
   \R_{\Gamma_i \Gamma}.
   \]
   Let $\R_{\lambda_i}:\textrm{span}( \{{\bf Q}^{(1)},\ldots,{\bf Q}^{(N)}\} )\to \textrm{span}({\bf Q}^{(i)})$ be a restriction operator that extracts the $k_i$ components corresponding to $\Omega_i$. As an example, let $N_E=\sum_{i=1}^N k_i$ be the total number of eigenvectors selected over all subdomains. The vector $\vec{u}=[u_{1,1},\ldots, u_{1,k_1},\ldots,u_{N,1},\ldots, u_{N,k_N}]^T$ has $N_E$ components, and the vector $\R_{\lambda_i}\vec{u}=[u_{i,1},\ldots, u_{i,k_i}]^T$ has $k_i$ components. The coarse matrix ${\bf A}_0$ can be assembled using information from the generalized eigenvalue problem on each subdomain. We define the following matrices: 
   \begin{equation*}
   \begin{alignedat}{3}
   &{\bf B}_{\Gamma\Gamma} = \sum_{i=1}^N \R_{\Gamma_i \Gamma}^T {\bf B}_{\Gamma\Gamma}^{(i)} \R_{\Gamma_i \Gamma},&&
   \quad{\bf U} =  \sum_{i=1}^N \R_{\Gamma_i \Gamma}^T {\bf B}_{\Gamma\Gamma}^{(i)}{\bf Q}^{(i) } \R_{\lambda_i}, \\ 
   &{\bf D} =  \sum_{i=1}^N \R_{\lambda_i}^T {\bf D}^{(i)} \R_{\lambda_i}, &&
   \quad{\bf C} =  \sum_{i=1}^N \R_{\lambda_i}^T 
   ({\bf Q}^{(i)^T} {\bf B}_{\Gamma\Gamma}^{(i)}{\bf Q}^{(i)})^{-1}
    \R_{\lambda_i}, \\
    &{\bf P} =  \sum_{i=1}^N \R_{I_i I}^T 
    {\bf P}^{(i)} 
    \R_{\lambda_i}.
    \end{alignedat}
    \end{equation*}
    
    The coarse problem can then be expressed as
    \begin{align} 
     {\bf A}_0  \vec{u}_\Gamma =
    \R_0 {\bf A} \R_0^T\vec{u}_\Gamma = \R_0\vec{b}
\implies    
    ({\bf B}_{\Gamma\Gamma} - {\bf U} {\bf D}{\bf C} {\bf U}^T) \vec{u}_\Gamma = \vec{b}_\Gamma + {\bf U} {\bf C}{\bf P}^T \vec{b}_I. 
    \label{eq:coarse_operator}
    \end{align}
    As such, no Galerkin triple product of the form $\R_0 {\bf A} \R_0^T$ is required to form the coarse space operator.
    

\subsection{Preconditioner application} \label{sec:precond1}
With the definition of the coarse space in \Cref{sec:coarse_space}, we can define the two-level additive Schwarz preconditioner for the SIPG scheme. The additive Schwarz method is not used as a stand-alone solver, but rather a preconditioner for a Krylov subspace method \cite{saad2003iterative} (see \Cref{alg:two}).
\begin{algorithm}
\caption{The preconditioned conjugate gradient method}\label{alg:two}
\textbf{Input: }{${\bf A}, \vec{x}_0,\vec{b}, \texttt{tol}$}
\\
\textbf{Output: }{$\vec{y}$ (where ${\bf A}\vec{y} \approx \vec{b} $)} 
\\
$\vec{r}_0 := \vec{b}- {\bf A} \vec{x}_0$
\\
Solve ${\bf M} \vec{z}_0 = \vec{r}_0$
\\
$\vec{p}_0 = \vec{z}_0$
\\
$n=0$
\\
\textbf{while} { $\|\vec{r}_{n}\|_2 > \texttt{tol}$ }
\\ 
\indent
$\quad\quad\quad \gamma_n : = \frac{\vec{r}_n^T\vec{z}_n}{\vec{p}_n^T{\bf A}\vec{p}_n}$
\\
\indent
$\quad\quad\quad\vec{x}_{n+1} := \vec{x}_{n } + \gamma_n  \vec{p}_n$
\\
\indent
$\quad\quad\quad\vec{r}_{n+1} := \vec{r}_{n } - \gamma_n  {\bf A}\vec{p}_n$
\\
\indent
 $\quad\quad\quad$ \textbf{if} $\|\vec{r}_{n+1}\|_2 <= \texttt{tol}$  \textbf{then} exit loop \textbf{end if}
\\
\indent
$\quad\quad\quad$ Solve ${\bf M} \vec{z}_{n+1} = \vec{r}_{n+1} $
\\
\indent
$\quad\quad\quad \beta_n := \frac{\vec{r}_{n+1}^T\vec{z}_{n+1}}{\vec{r}^T_n\vec{z}_n}$
\\
\indent
$\quad\quad\quad \vec{p}_{n+1} := \vec{z}_{n+1} + \beta_n \vec{p}_{n}$
\\
\indent
$\quad\quad\quad n:=n+1$
\\
\indent
\textbf{end while}
\end{algorithm}
The preconditioner matrix $\bf M$ is not formed in practice. Instead, the application of the additive Schwarz preconditioner is performed as follows. Given a residual $\vec{r}_{n+1}$, we obtain $\vec{z}_{n+1}$ (see \Cref{alg:two}) by
\[
\vec{z}_{n+1} = 
\R_0^T {\bf A}_0^{-1} \R_0\vec{r}_{n+1}
+
\sum_{i=1}^N \R^{(i)^T } ({\bf A}_{II}^{(i)})^{-1}\R^{(i)} \vec{r}_{n+1},
\]
where we observe that the coarse grid correction and local solvers are completely data parallel.
 
\subsection{Extension to NIPG and IIPG}\label{sec:non_sym}
The NIPG and IIPG discretizations give rise to a nonsymmetric matrix, which causes challenges when designing preconditioners. To circumvent this issue for NIPG and IIPG, we instead consider the symmetric part of the matrix. That is, if $ {\bf A}$ is the DG discretization matrix, we build the preconditioner with respect to the matrix ${\bf A}_{\texttt{sym}} = (1/2)({\bf A} + {\bf A}^T)$. The idea of preconditioning NIPG and IIPG using the symmetric part of the matrix has been explored in several works \cite{ayuso2014multilevel,ayuso2009uniformly,johannsen2005symmetric}.

The coarse space and local problems are formed using the symmetrized bilinear forms and matrix ${\bf A}_{\texttt{sym}}$. This allows us to reuse the generalized eigenvalue problem posed in \Cref{sec:coarse_space}. Moreover, the global coarse grid operator is rendered symmetric as a consequence. For the outer Krylov solver, we use either biconjugate gradient method (BICG) or generalized minimum residual method (GMRES) \cite{saad2003iterative}. We emphasize that the outer Krylov solver is applied to the original discretization matrix $\bf A$ (resulting from NIPG or IIPG), and ${\bf A}_{\texttt{sym}}$ is only utilized to form the preconditioner (local solvers and global coarse grid matrix).
 
%

\subsection{Further reduction of the coarse space} \label{sec:coarse_space2}
The coarse problem is defined on $V_h(\Gamma)$.  In terms of the global coarse matrix~\eqref{eq:coarse_operator} we have
\[
{\bf S}_\Gamma
= {\bf B}_{\Gamma\Gamma} - {\bf U}{\bf D}{\bf C}{\bf U}^T.
\]
We have a few options if the coarse problem is too large. The first option works with ${\bf B}_{\Gamma\Gamma}$ 
directly, leveraging the observation that ${\bf B}_{\Gamma\Gamma}$ is much sparser than ${\bf S}_\Gamma$. The second option considers searching for a coarse grid correction in the conforming continuous Galerkin (CG) auxiliary space $V_{\text{CG}}(\Gamma)=\{v\in C^0(\Omega): v|_E\in\mathcal{P}_1 ,~\forall E\in \mathcal{T}_h\}.$
\begin{enumerate} 
\item  The matrix ${\bf A}_{\Gamma\Gamma}$ is much sparser than ${\bf S}_\Gamma$. Hence,
\begin{align} 
{\bf S}_\Gamma^{-1} 
=
{\bf B}_{\Gamma\Gamma}^{-1} 
+
{\bf B}_{\Gamma\Gamma}^{-1} 
{\bf U}
(
{\bf C}^{-1}{\bf D}^{-1}
-
{\bf U}^T{\bf B}_{\Gamma\Gamma}^{-1}{\bf U} 
)^{-1}
{\bf U}^T{\bf B}_{\Gamma\Gamma}^{-1}.
\label{eq:woodbury}
\end{align}
The matrix $(
{\bf C}^{-1}{\bf D}^{-1}
-
{\bf U}^T{\bf B}_{\Gamma\Gamma}^{-1}{\bf U} 
)$ is a $N_E\times N_E$ matrix, where $N_E = \sum_{j=1}^{N } k_j$ (where $k_j$ eigenvalues from subdomain $j$).
\item Note that ${\bf B}_{\Gamma\Gamma}$ is much sparser than ${\bf S}_\Gamma$, but these two matrices have the same dimensions. We can considerably reduce the size of the coarse space by considering an auxiliary space $V_{\text{CG}}(\Gamma) \subset V_h(\Gamma)$. We define a mapping $\R_{\text{DG}}^{\text{CG}}:V_{\text{DG}}(\Gamma)\to V_{\text{CG}}(\Gamma)$ (boolean operator). This can be easily formed since we have access to the fine-grid mesh.  In more detail, if  $ V_{\text{CG}}(\Omega)= \text{span}(\psi_1,\ldots,\psi_{n_0}) $ and $ V_h(\Omega)= \text{span}(\phi_1,\ldots,\phi_{n}) $, then we can write any CG basis function $\psi_i$ in terms of a linear combination the DG basis functions $\phi_j$, because $ V_{\text{CG}}(\Omega) \subset V_h(\Omega)$:
\[
\psi_i = \sum_{j=1}^n  (\R_{\text{DG}}^{\text{CG}})_{i,j} \phi_j.
\]
We summarize the subtle points about this approach:
\begin{enumerate}
\item  Local solvers are defined in terms of the DG bilinear form.
\item  The intermediate coarse problem ${\bf S}_\Gamma$, is defined in terms of the DG bilinear form and the generalized eigenvalue problems.
\item The CG coarse problem is defined as $\tilde{{\bf S}}_\Gamma = R_{\text{DG}}^{\text{CG}} {\bf S}_\Gamma (R_{\text{DG}}^{\text{CG}})^T$. This linear system is significantly smaller than $ {\bf S}_\Gamma$.
\item After the NOSAS application, we apply $\nu_2$ post-relaxation steps (say weighted block Jacobi). This is to resolve any issues with the auxiliary space projection.
\end{enumerate} 


\item Instead of using the Galerkin triple product $\tilde{{\bf S}}_\Gamma = R_{\text{DG}}^{\text{CG}} {\bf S}_\Gamma (R_{\text{DG}}^{\text{CG}})^T$, we can consider rediscretization. Meaning, the problem \eqref{eq:model} is discretized using the standard $\mathcal{P}_1$ continuous finite element method, and the resulting matrix is used as the coarse problem.

\item If the bilinear forms affiliated with the generalized eigenvalue problems are selected in a specific way (e.g., Dirichlet-to-Neumann map), it is possible to obtain a robust spectral coarse space defined on the eigenfunction space $\text{span}\{ {\bf Q}^{(1)},\ldots, {\bf Q}^{(N)} \}$ (as opposed to the interface space, $V_h(\Gamma)$) \cite{liu2020two,nataf2010two,nataf2011coarse}.

\item Multi-level additive Schwarz. Recursively introduce additional levels to extend the scalability of the method \cite{bastian2022multilevel,yu2021additive}.
\end{enumerate}

\section{Theoretical results} \label{sec:theory}
To estimate the condition number of the proposed preconditioner, we follow the abstract theory of Schwarz Methods \cite{spillane2014abstract,toselli2004domain}. Essentially, this reduces to the verification of three key assumptions, which we state below for convince.
\begin{assumption}[Strengthened Cauchy-Schwarz Inequalities] \label{assumption1}
There exists constants $0\le \varepsilon_{ij}\le 1$, $1\le i \le N$ such that
\[
|a(\R_i^T u_i,\R_j^T u_j)| \le  \varepsilon_{ij}a(\R_i^T u_i,\R_i^T u_i)^{1/2} a(\R_j^T u_j,\R_j^T u_j)^{1/2},
\]
for $u_i\in V_h^i$, and $u_j\in V_h^j.$ 
\end{assumption} 

\begin{assumption}[Local Stability] \label{assumption2}
There exists a constant $\omega>0$ such that
\[
a(\R_i^T u_i,\R_i^T u_i) \le \omega a_i(  u_i, u_i),\quad \forall u_i \in V_h^i,\quad 0\le i\le N.
\] 
\end{assumption} 

\begin{assumption}[Stable Decomposition] \label{assumption3}
	There exists a constant $C_0$, such that every $u\in V_h$ admits a decomposition
	\[
	u = \sum_{i=0}^N \R_i^T u_i,\quad u_i \in V_h^i,\quad 0\le i \le N,
	\]
	that satisfies
	\[
	\|\R_0^T u_0\|_{DG}^2 +\sum_{i=1}^N \|\R_i^T u_i\|_{DG}^2 \le C_0^2 \| u\|_{DG}^2,
	\]
    where $\|\cdot \|_{DG}$ is defined in \eqref{eqn_dg_norm}.
\end{assumption}

\begin{theorem}(Condition number estimate \cite{toselli2004domain}) \label{thm_condition_numb}
	Let \Cref{assumption1}, \Cref{assumption2}, and \Cref{assumption3} be satisfied. Then, the condition number of the additive Schwarz operator satisfies
	\[
	\kappa( P ) \le C_0^2\omega (\rho({\bf E}) + 1),
	\]
	where ${\bf E}_{ij} = \varepsilon_{ij}$, and the spectral radius of $\bf E$ is given by $\rho({\bf E})$.
\end{theorem}

 
 


To establish the stable decomposition property (Assumption \ref{assumption3}), we first prove two lemmas which relate the spectrum of the coarse space to the Schur complement. 
\begin{lemma} \label{lemma_stab1}
Define ${\bf \Pi}^{(i)}  : V_h(\Gamma_i) \to \textrm{span}({\bf Q}^{(i)}) $ so that
\[
{\bf \Pi}^{(i)} = {\bf Q}^{(i)} ( ({\bf Q}^{(i)})^T {\bf B}_{\Gamma\Gamma}^{(i)} {\bf Q}^{(i)})^{-1} ({\bf Q}^{(i)})^T.
\]
For all $ u_{\Gamma_i}, v_{\Gamma_i} \in V_h(\Gamma_i)$ for $i=1,2,\ldots, N$ set
\[
a_0^{(i)}(u_{\Gamma_i}, v_{\Gamma_i} )=
 v_{\Gamma_i}^T
\bigg(
{\bf B}_{\Gamma\Gamma}^{(i)}
-
{\bf B}_{\Gamma\Gamma}^{(i)}
{\bf D}^{(i)}
 ( ({\bf Q}^{(i)})^T {\bf B}_{\Gamma\Gamma}^{(i)} {\bf Q}^{(i)})^{-1}
 ({\bf Q}^{(i)})^T
 {\bf B}_{\Gamma\Gamma}^{(i)}
 \bigg)
 u_{\Gamma_i}.
\]
Then, for each $i=1,2,\ldots,N$, we have
\[
a_0^{(i)}(u_{\Gamma_i}, v_{\Gamma_i} ) =
({\bf \Pi}^{(i)}v_{\Gamma_i})^T {\bf S}^{(i)}({\bf \Pi}^{(i)}u_{\Gamma_i})
+
(v_{\Gamma_i} - {\bf \Pi}^{(i)}v_{\Gamma_i})^T
 {\bf B}_{\Gamma\Gamma}^{(i)}
 (u_{\Gamma_i} - {\bf \Pi}^{(i)}u_{\Gamma_i}).
\]
\end{lemma}
\begin{proof}
	We decompose $u_{\Gamma_i},v_{\Gamma_i}\in V_h(\Gamma_i)$ in terms of components that are members of the eigenspace $\text{span}({\bf Q}^{(i)})$ and its orthogonal complement $\text{span}({\bf Q}^{(i)})^\perp$. Let $u_{\Gamma_i} = u_1 + u_2$, $v_{\Gamma_i} = v_1 + v_2$ where $u_1={\bf \Pi}^{(i)} u_{\Gamma_i},v_1={\bf \Pi}^{(i)} v_{\Gamma_i}\in \text{span}({\bf Q}^{(i)})$ and 
	$u_2 = ({\bf I} - {\bf \Pi}^{(i)})u_{\Gamma_i} ,
	v_2
	= ({\bf I} - {\bf \Pi}^{(i)})v_{\Gamma_i} \in\text{span}({\bf Q}^{(i)})^\perp.
	$
	Upon inspecting $a^{(i)}(u_{\Gamma_i}, v_{\Gamma_i} )$, we find
	\begin{align*}
	a^{(i)}(u_{\Gamma_i}, v_{\Gamma_i} ) &= 
	v_1^T\bigg(
{\bf B}_{\Gamma\Gamma}^{(i)}
-
{\bf B}_{\Gamma\Gamma}^{(i)}
{\bf D}^{(i)}
 ( ({\bf Q}^{(i)})^T {\bf B}_{\Gamma\Gamma}^{(i)} {\bf Q}^{(i)})^{-1}
 ({\bf Q}^{(i)})^T
 {\bf B}_{\Gamma\Gamma}^{(i)}
 \bigg)u_1 
 + 
 v_2^T
 {\bf B}_{\Gamma\Gamma}^{(i)}u_2 ,
	\end{align*}
	since the pairs $u_1,v_2$ and $u_2,v_1$ are orthogonal in the ${\bf A}_{\Gamma\Gamma}^{(i)}$ inner product, and from the projection ${\bf \Pi}^{(i)}$ we deduce $({\bf Q}^{(i)})^T{\bf B}_{\Gamma\Gamma}^{(i)}({\bf I} - {\bf \Pi}^{(i)}) ={\bf 0}.$
	
	The generalized eigenvalue problem then asserts that for any $\xi \in  \textrm{span}({\bf Q}^{(i)})$, one has
	\[
	\xi^T {\bf S}^{(i)} \xi = \lambda \xi^T {\bf B}_{\Gamma\Gamma}^{(i)}\xi
=
\xi^T
\bigg(
{\bf B}_{\Gamma\Gamma}^{(i)}
-
{\bf B}_{\Gamma\Gamma}^{(i)}
{\bf D}^{(i)}
 ( ({\bf Q}^{(i)})^T {\bf B}_{\Gamma\Gamma}^{(i)} {\bf Q}^{(i)})^{-1}
 ({\bf Q}^{(i)})^T
 {\bf B}_{\Gamma\Gamma}^{(i)}
 \bigg)
 \xi	
	.
	\]
	From the above relation, and observing that $u_1,v_1 \in  \textrm{span}({\bf Q}^{(i)})$  we conclude
	\[
	v_1^T{\bf S}^{(i)}u_1 
=
v_1^T
\bigg(
{\bf B}_{\Gamma\Gamma}^{(i)}
-
{\bf B}_{\Gamma\Gamma}^{(i)}
{\bf D}^{(i)}
 ( ({\bf Q}^{(i)})^T {\bf B}_{\Gamma\Gamma}^{(i)} {\bf Q}^{(i)})^{-1}
 ({\bf Q}^{(i)})^T
 {\bf B}_{\Gamma\Gamma}^{(i)}
 \bigg)
u_1	 ,
	\]	
	and the result follows.
\end{proof}

\begin{lemma}  \label{lemma_stab2}
	Set $\eta$ to be the cutoff eigenvalue threshold. Then, for all $u_\Gamma\in V_0$, we have
	\[
	a_0(u_\Gamma,u_\Gamma)
	=
	\sum_{i=1}^N a_0^{(i)}( {\bf R}_{\Gamma_i\Gamma}^T u_\Gamma, {\bf R}_{\Gamma_i\Gamma} u_\Gamma)
	\le 
	\eta^{-1}
	\sum_{i=1}^N  u_\Gamma ^T{\bf R}_{\Gamma_i\Gamma}^T {\bf S}^{(i)} {\bf R}_{\Gamma_i\Gamma} u_\Gamma 
	=
	\eta^{-1}u_\Gamma^T {\bf S} u_\Gamma.
	\]
\end{lemma}
\begin{proof}
Let $u^{(i)} = {\bf \Pi}^{(i)} {\bf R}_{\Gamma_i\Gamma}  u_\Gamma$ and $v^{(i)} =  ({\bf I} - {\bf \Pi}^{(i)}){\bf R}_{\Gamma_i\Gamma}   u_\Gamma = {\bf R}_{\Gamma_i\Gamma}u_\Gamma -u^{(i)},$ so that $u^{(i)}$ and $v^{(i)}$ are orthogonal in the appropriate inner product. The generalized eigenvalue problem allows us to establish for any $v_{\Gamma_i}\in \textrm{span}({\bf Q}^{(i)})^\perp$
\[
v_{\Gamma_i}^T {\bf S}^{(i)} v_{\Gamma_i}
\ge 
\eta 
v_{\Gamma_i}^T {\bf B}_{\Gamma\Gamma}^{(i)} v_{\Gamma_i}.
\]
Similarly, for any $v_{\Gamma_i}\in \textrm{span}({\bf Q}^{(i)}),$ we have
\[
v_{\Gamma_i}^T {\bf S}^{(i)} v_{\Gamma_i}
<
\eta 
v_{\Gamma_i}^T {\bf B}_{\Gamma\Gamma}^{(i)} v_{\Gamma_i}.
\]
Using Lemma~\ref{lemma_stab1}, and $0<\eta<1$, we obtain
\begin{align*}
a_0^{(i)}({\bf R}_{\Gamma_i\Gamma}  u_\Gamma ,{\bf R}_{\Gamma_i\Gamma}  u_\Gamma )
	&=
	(u^{(i)})^T{\bf S} u^{(i)} + (v^{(i)})^T{\bf B}_{\Gamma\Gamma} v^{(i)} 
	\\
	&\le (u^{(i)})^T{\bf S} u^{(i)} + \eta^{-1}(v^{(i)})^T {\bf S} v^{(i)}
	\\
	&\le  \eta^{-1}(u^{(i)})^T{\bf S} u^{(i)} + \eta^{-1}(v^{(i)})^T {\bf S} v^{(i)}	
\\
	&=
 \eta^{-1}({\bf R}_{\Gamma_i\Gamma}  u_\Gamma )^T {\bf S}^{(i)} {\bf R}_{\Gamma_i\Gamma}  u_\Gamma ,  
\end{align*}
which holds for all subdomains $i=1,2,\ldots,N$. We then immediately find
\begin{align*}
a_0(   u_\Gamma ,  u_\Gamma )
 &= 
 \sum_{i=1}^N
 a_0^{(i)}(  {\bf R}_{\Gamma_i\Gamma}  u_\Gamma ,  {\bf R}_{\Gamma_i\Gamma}  u_\Gamma )
 \\
 &\le 
 \eta^{-1} 
  \sum_{i=1}^N
 ({\bf R}_{\Gamma_i\Gamma}  u_\Gamma )^T {\bf S}^{(i)} {\bf R}_{\Gamma_i\Gamma}  u_\Gamma 
 \\
 &= \eta^{-1}   u_\Gamma^T{\bf S}u_\Gamma.
\end{align*}
The last equality follows from the choice of $A^{(i)}$ such that ${\bf A} = \sum_{i=1}^N \R^{(i)^T} {\bf A}^{(i)} \R^{(i) }$.
\end{proof}

\begin{theorem}[Stable Decomposition]
Suppose $u\in V_h$, where $u = \sum_{i=0}^{N} \R_i^T u_i,$ $u_i\in V_h^i$ for $0\le i \le N$. Then, there exists a constant $C_0$ which is independent of $\K$, the subdomain size $H$ and the mesh $h$ such that
	\begin{equation}
	\sum_{i=0}^N {a}_i(u_i,u_i) \le C_0^2 a (u ,u ),
	\end{equation}
	and $C_0^2 = 2 + 3\eta^{-1}$.
\end{theorem}
\begin{proof}
	We can write $\sum_{i=1}^N\R_i^T u_i = u - \R_0^T u_0$; then the Cauchy-Schwarz inequality yields
\begin{align*}
	\sum_{i=0}^N a_i(u_i,u_i) &= a_0(u_0,u_0)+ \sum_{i=1}^N a_i(u_i,u_i)
	\\
	&= a(\R_0^T u_0,\R_0^T u_0) + \sum_{i=1}^N a (\R_i^T u_i,\R_i^T u_i)
    \\
	&= a(\R_0^T u_0,\R_0^T u_0) +  a ( u - \R_0^T u_0, u - \R_0^T u_0)
	\\
	&\le 3  a(\R_0^T u_0,\R_0^T u_0) + 2 a (u,u).                                      
\end{align*}	
We next bound the term $a(\R_0^T u_0,\R_0^T u_0)$. From Lemma~\ref{lemma_stab2} we have
\begin{align*} 
a(\R_0^T u_0,\R_0^T u_0) &\le \eta^{-1} s(u_0,u_0)
\\
 &= \eta^{-1} a( {\bf \mathcal{H}} u_0,{\bf \mathcal{H}} u_0)
 \\
 &\le \eta^{-1} a( u,u),
\end{align*}
because of the minimum energy principle for functions in the discrete harmonic space $V_h(\Gamma)$ \cite{brenner2007mathematical}. These relations enable us to arrive at
\begin{align}
	\sum_{i=0}^N a_i(u_i,u_i)  &\le 3  a(\R_0^T u_0,\R_0^T u_0) + 2 a (u,u)
	\\
	&\le (2 +3\eta^{-1} ) a (u,u),      
\end{align}	
and the result follows.
\end{proof}
  
Assumption \ref{assumption1} is straightforward to verify with $\rho({\bf E}) = 1$ since the spaces $V_i$ are disjoint. 
  
Assumption  \ref{assumption2} follows with $\omega=1$ for  $i=1,\dots,N$ since
we use exact local bilinear forms. Now we consider the case of the coarse space. The case 
${\bf B}_{\Gamma \Gamma} = {\bf A}_{\Gamma \Gamma}$ implies $\omega=1$ since we use exact global bilinear form. In the case ${\bf B}_{\Gamma \Gamma} = {\bf \hat{A}}_{\Gamma \Gamma}$ we use ${\bf{A}}_{\Gamma \Gamma} \leq 3 {\bf \hat{A}}_{\Gamma \Gamma}$ because there are no more than two off-diagonal blocks of ${\bf{A}}_{\Gamma \Gamma}^{(i)}$ on the same block row, hence $\omega=3$. 
In the case ${\bf B}_{\Gamma \Gamma} = {\bf \hat{S}}_{\Gamma \Gamma}$ we use that  $ {\bf S}_{\Gamma \Gamma} \leq \omega  {\bf \hat{S}}_{\Gamma \Gamma}$ where $\omega$ is maximum number of edges plus corners of the $\Gamma_i$

Together with the Assumption  \ref{assumption3},\ Theorem \ref{thm_condition_numb} allows us to conclude that the condition number of the two-level additive Schwarz operator is independent of the number of subdomains as well as the highly varying coefficients.
   
\section{Numerical experiments} \label{sec:num_ex}
The results of \Cref{sec:theory} are verified and validated here. In addition, the preconditioner for nonsymmetric IPDG (\Cref{sec:non_sym}) are examined here. Some of the coarse spaces from \Cref{sec:coarse_space2} are also compared. For all numerical experiments, we fix the user-defined penalty parameter $\sigma = 15$ 
(see \Cref{eq:dg_penalty}).


To increase transparency, reproducibility, and for comparison purposes, the underlying benchmarks used in \Cref{sec:num_ex_1},  \Cref{sec:num_ex_2}, and \Cref{sec:num_ex_3} are adapted from \cite{dryja2014additive}. For all numerical experiments we assume homogeneous Dirichlet boundary conditions and set the forcing function $f\equiv 1$ in \eqref{eq:model}.

\subsection{Jump coefficient aligned with mesh} \label{sec:num_ex_1}  
    We take the unit square and partition it into a $4\times 4$ square grid. A red–white checkerboard coloring of this partitioning is made. The coefficient $\K$ is piecewise constant such that it equals $\rho_1$ in the red regions, and one on the white regions:
    \[
    \K =
    \begin{cases}
    \rho_{1}  
     &\mbox{ if } E\in \mathcal{T}_h,~~ E\in \texttt{red}
    \\
 1 
       &\mbox{ if }E\in \mathcal{T}_h,~~E\in  \texttt{white}
    \end{cases}.
    \]
The $4\times 4$ square is further subdivided into a mesh of right triangles (see \Cref{fig:num_ex_1}).
\begin{figure}[htb!]
    \centering
    \subfloat[\centering $h=1/8$]{{\includegraphics[scale=0.4]{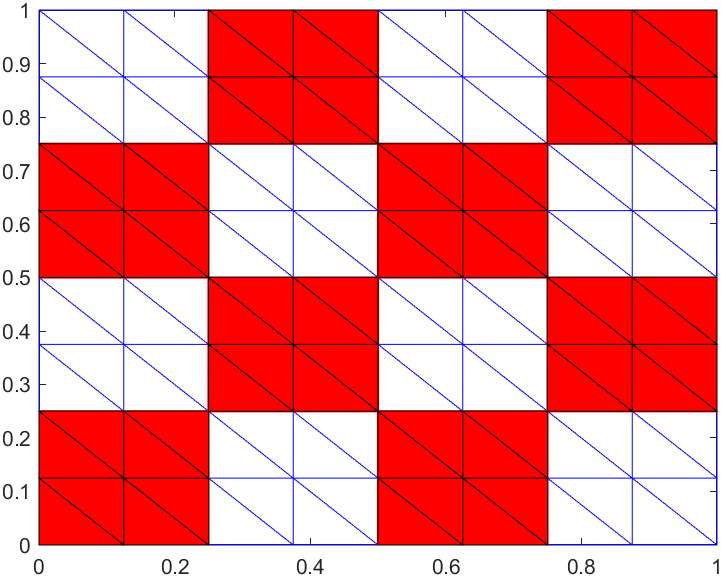} } }%
    \qquad
    \subfloat[\centering $h=1/16$]{{\includegraphics[scale=0.4]{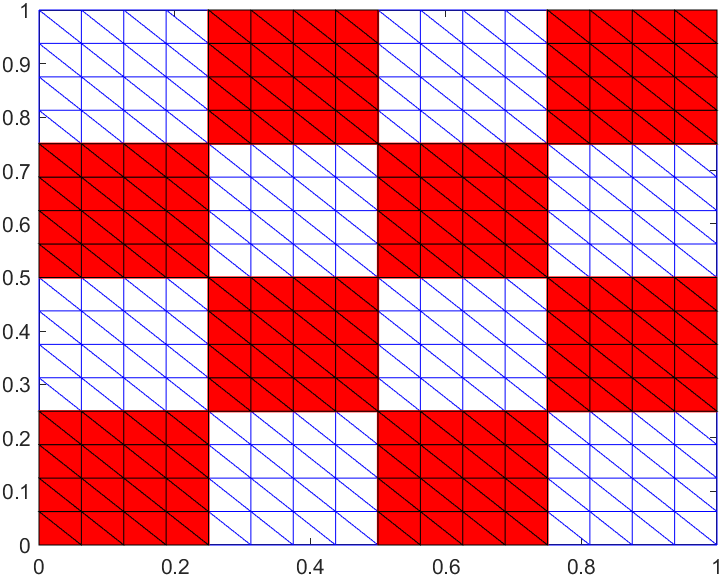} }
    }%
    \caption{Checkerboard coloring for the $4\times 4$ partitioning.
    }%
    \label{fig:num_ex_1}%
\end{figure} 
We vary $\rho_1 = 10^j$ for $j=0,1,2,\ldots,6$. For the eigenvalue threshold (see \Cref{sec:coarse_space}).  The computational results are provided in \Cref{tab:tab_ex_1}. We keep track of the number of preconditioned Krylov steps required to reduce the relative residual smaller than $10^{-6}$. The terms in parenthesis are the estimated condition numbers of the preconditioned problem. We can see that as the jump in coefficients grows, the number of iterations remains small and the condition number is bounded. Moreover, as the ratio $\frac{H}{h}$ increases, the iterations and condition numbers remain similarly uniform. It is evident that the additive Schwarz preconditioner is very robust in the situation where the jump in $\K$ aligns with the mesh.

\begin{table}[htbp]
\centering
\scriptsize
\setlength{\tabcolsep}{4pt}
\renewcommand{\arraystretch}{1.15}
\begin{tabular}{c c c c c c c c c c}
\toprule
\(H/h\) & \(N_E\) & \% ret.
& \(10^0\) & \(10^1\) & \(10^2\) & \(10^3\)
& \(10^4\) & \(10^5\) & \(10^6\) \\
\midrule
4
& 188 & 25.07
& 12 (9.17) & 10 (4.74) & 10 (4.33) & 10 (4.38)
& 10 (4.38) & 10 (4.38) & 10 (4.38) \\

8
& 372 & 24.51
& 8 (2.92) & 7 (2.17) & 7 (2.12) & 7 (2.15)
& 7 (2.15) & 7 (2.15) & 7 (2.15) \\

16
& 736 & 24.10
& 8 (3.87) & 10 (4.61) & 10 (5.40) & 10 (5.62)
& 10 (4.96) & 9 (5.14) & 9 (5.29) \\

32
& 1476 & 24.09
& 9 (3.32) & 12 (4.65) & 10 (6.60) & 10 (5.45)
& 9 (4.92) & 9 (4.99) & 8 (4.75) \\
\bottomrule
\end{tabular}
\caption{Results of the numerical experiments from \Cref{sec:num_ex_1} (isotropic jump coefficients aligned with mesh). The local selection rule retains approximately
the lowest \(24\%\) of the generalized eigenmodes on each subdomain. The columns
\(10^j\) correspond to \(\rho_1=10^j\). Entries report Krylov iterations and,
in parentheses, the estimated condition number \({\kappa}\). The quantities
\(N_E\) and \(\%\) retained denote the total number and percentage of local
generalized eigenvectors retained over all subdomains.} 
\label{tab:tab_ex_1}
\end{table}

\subsection{Jump coefficient not aligned with mesh} \label{sec:num_ex_2} 
We repeat the experiments in \Cref{sec:num_ex_1}, but use an initial $3\times 3$ partition of the unit square. This partitioning is only used to define the checkerboard pattern for the coefficient $\K$. The triangular mesh is formed by first subdividing the unit square into a grid of smaller $(2^{n_H})^2$ squares. Each of these squares are further divided into equal area right triangles. This way, the resulting triangular mesh is not aligned with the jumps in the coefficients $\K$.

We take $H=1/4$, and pick $h$ to fix a ratio $\frac{H}{h}\in\{4,8,16\}$. For this cost-normalized comparison, we use a fixed-fraction spectral selection: on each subdomain, approximately the lowest \(24\%\) of the local generalized eigenvalues are retained in the coarse space.




\begin{table}[htbp]
\centering
\scriptsize
\setlength{\tabcolsep}{4pt}
\renewcommand{\arraystretch}{1.15}
\begin{tabular}{c c c c c c c c c c}
\toprule
\(H/h\) & \(N_E\) & \% ret.
& \(10^0\) & \(10^1\) & \(10^2\) & \(10^3\)
& \(10^4\) & \(10^5\) & \(10^6\) \\
\midrule
4
& 129 & 34.31
& 9 (4.02) & 12 (9.16) & 13 (11.54) & 13 (11.59)
& 13 (11.51) & 13 (11.49) & 13 (11.49) \\

8
& 187 & 24.61
& 8 (2.75) & 10 (4.69) & 11 (6.42) & 11 (6.48)
& 11 (6.49) & 11 (6.49) & 11 (6.49) \\

16
& 369 & 24.15
& 8 (3.75) & 9 (4.03) & 10 (5.33) & 9 (4.80)
& 9 (5.04) & 9 (4.59) & 8 (3.06) \\

32
& 739 & 24.12
& 9 (3.96) & 11 (4.57) & 11 (5.71) & 11 (6.02)
& 8 (3.51) & 8 (3.50) & 8 (3.50) \\
\bottomrule
\end{tabular}
\caption{Results of the numerical experiments from \Cref{sec:num_ex_2} (jump coefficients unaligned with mesh).
The \(H/h=4\) row uses a larger local fixed-fraction selection, retaining about
\(34.3\%\) of the generalized eigenmodes, while the remaining rows retain about
\(24\%\). The columns \(10^j\) correspond to \(\rho_1=10^j\). Entries report
Krylov iterations and, in parentheses, the estimated condition number
\({\kappa}\). The quantities \(N_E\) and \(\%\) retained denote the total
number and percentage of local generalized eigenvectors retained over all
subdomains.}
\label{tab:tab_ex_2}
\end{table}

\Cref{tab:tab_ex_2} has results of the experiment. 
We can see that as the jump in coefficients grows, the number of iterations remains small, and the condition number remains fixed. Moreover, as the ratio $\frac{H}{h}$ increases, the iterations and condition numbers remain nearly uniformly controlled. Similarly to the results of \Cref{sec:num_ex_1}, we see that the preconditioner does not deteriorate if the jump coefficients are unaligned with the mesh or domain decomposition.

\subsection{Comparisons for different choices of \(B_{\Gamma\Gamma}^{(i)}\)}

We repeat the unaligned checkerboard coefficient experiment from
\Cref{sec:num_ex_2}. Both the local generalized eigenvalue problem and the
corresponding coarse-grid matrix depend on the choice of
\(B_{\Gamma\Gamma}^{(i)}\). We consider two additional choices, namely
\(B_{\Gamma\Gamma}^{(i)}=\operatorname{blkdiag}(S_{\Gamma\Gamma}^{(i)})\)
and
\(B_{\Gamma\Gamma}^{(i)}=\operatorname{blkdiag}(A_{\Gamma\Gamma}^{(i)})\).

\subsubsection{Case of
\(B_{\Gamma\Gamma}^{(i)}
=\operatorname{blkdiag}(S_{\Gamma\Gamma}^{(i)})\)}

The choice
\(B_{\Gamma\Gamma}^{(i)}
=\operatorname{blkdiag}(S_{\Gamma\Gamma}^{(i)})\)
results in a very robust preconditioner. However, the local problems generally
require the formation and storage of the local Schur matrices
\(S_{\Gamma\Gamma}^{(i)}\), which can be expensive.

\Cref{tab:nosas_pencil_compare_blkS} displays the results of the NOSAS
preconditioner applied to the test problem from \Cref{sec:num_ex_2} with
\(B_{\Gamma\Gamma}^{(i)}
=\operatorname{blkdiag}(S_{\Gamma\Gamma}^{(i)})\).
Significantly fewer eigenvectors are required to obtain performance comparable
to that of the choice
\(B_{\Gamma\Gamma}^{(i)}=A_{\Gamma\Gamma}^{(i)}\); see
\Cref{tab:tab_ex_2}.

\begin{table}[htbp!]
\centering
\scriptsize
\setlength{\tabcolsep}{4pt}
\renewcommand{\arraystretch}{1.15}
\begin{tabular}{c c c c c c c c c c}
\toprule
\(H/h\) & \(N_E\) & \% ret.
& \(10^0\) & \(10^1\) & \(10^2\) & \(10^3\)
& \(10^4\) & \(10^5\) & \(10^6\) \\
\midrule
4
& 9 & 2.39
& 10 (6.10) & 11 (6.77) & 11 (7.83) & 10 (8.22)
& 10 (8.58) & 10 (8.62) & 10 (8.62) \\

8
& 10 & 1.32
& 10 (8.63) & 11 (9.45) & 11 (12.33) & 11 (11.86)
& 11 (11.93) & 11 (11.94) & 11 (11.94) \\

16
& 19 & 1.24
& 10 (5.37) & 10 (4.66) & 9 (4.53) & 9 (4.75)
& 9 (4.76) & 9 (4.76) & 9 (4.76) \\

32
& 34 & 1.11
& 9 (5.19) & 9 (5.02) & 9 (5.00) & 9 (5.16)
& 9 (5.18) & 9 (5.18) & 9 (5.18) \\
\bottomrule
\end{tabular}
\caption{NOSAS results for the unaligned checkerboard coefficient test using
\(B_{\Gamma\Gamma}^{(i)}
=\operatorname{blkdiag}(S_{\Gamma\Gamma}^{(i)})\).}
\label{tab:nosas_pencil_compare_blkS}
\end{table}

As explained in \Cref{sec:coarse_space}, the choice
\(B_{\Gamma\Gamma}^{(i)}=A_{\Gamma\Gamma}^{(i)}\) ensures that the corresponding
local generalized eigenvalues lie in the interval \([0,1]\). This is no longer
true in general for other choices of \(B_{\Gamma\Gamma}^{(i)}\). We therefore
also report the maximum and minimum eigenvalues for
\(B_{\Gamma\Gamma}^{(i)}
=\operatorname{blkdiag}(S_{\Gamma\Gamma}^{(i)})\)
before any eigenvectors are removed.

These values are displayed in
\Cref{tab:full-local-eig-extremes-blkS}. The maximum eigenvalue is greater than
one and approaches two as the coefficient contrast increases. Each floating
subdomain has one zero eigenvalue, up to numerical roundoff.

\begin{table}[htbp]
\centering
\scriptsize
\caption{Maximum and minimum local generalized eigenvalues for
\(B_{\Gamma\Gamma}^{(i)}
=\operatorname{blkdiag}(S_{\Gamma\Gamma}^{(i)})\). Each entry is
\(\lambda_{\max}\,(\lambda_{\min})\).}
\label{tab:full-local-eig-extremes-blkS}

\setlength{\tabcolsep}{4pt}

\scalebox{0.95}{%
\begin{tabular}{c|cccc}
\hline
\(H/h\) & \(10^0\) & \(10^1\) & \(10^2\) & \(10^3\) \\
\hline
\(4\)
& \(1.768\,(-1.55{\times}10^{-15})\)
& \(1.897\,(-8.79{\times}10^{-15})\)
& \(1.985\,(-5.16{\times}10^{-14})\)
& \(1.998\,(8.49{\times}10^{-15})\)
\\
\(8\)
& \(1.858\,(-1.28{\times}10^{-13})\)
& \(1.933\,(1.98{\times}10^{-14})\)
& \(1.990\,(-2.31{\times}10^{-14})\)
& \(1.999\,(2.24{\times}10^{-14})\)
\\
\(16\)
& \(1.908\,(-2.71{\times}10^{-14})\)
& \(1.957\,(-6.35{\times}10^{-14})\)
& \(1.993\,(-2.15{\times}10^{-13})\)
& \(1.999\,(-2.38{\times}10^{-13})\)
\\
\(32\)
& \(1.936\,(9.16{\times}10^{-13})\)
& \(1.972\,(9.91{\times}10^{-13})\)
& \(1.995\,(-2.37{\times}10^{-12})\)
& \(2.000\,(-8.75{\times}10^{-13})\)
\\
\hline\hline
\(H/h\) & \(10^4\) & \(10^5\) & \(10^6\) & \\
\hline
\(4\)
& \(2.000\,(8.39{\times}10^{-16})\)
& \(2.000\,(-1.83{\times}10^{-14})\)
& \(2.000\,(6.19{\times}10^{-15})\)
&
\\
\(8\)
& \(2.000\,(-4.99{\times}10^{-14})\)
& \(2.000\,(-3.59{\times}10^{-14})\)
& \(2.000\,(-1.75{\times}10^{-14})\)
&
\\
\(16\)
& \(2.000\,(-2.97{\times}10^{-14})\)
& \(2.000\,(-8.31{\times}10^{-14})\)
& \(2.000\,(-2.08{\times}10^{-12})\)
&
\\
\(32\)
& \(2.000\,(-6.52{\times}10^{-12})\)
& \(2.000\,(-8.92{\times}10^{-13})\)
& \(2.000\,(4.01{\times}10^{-13})\)
&
\\
\hline
\end{tabular}%
}
\end{table}

\subsubsection{Case of
\(B_{\Gamma\Gamma}^{(i)}
=\operatorname{blkdiag}(A_{\Gamma\Gamma}^{(i)})\)}

We next consider
\(B_{\Gamma\Gamma}^{(i)}
=\operatorname{blkdiag}(A_{\Gamma\Gamma}^{(i)})\),
which is less expensive than
\(B_{\Gamma\Gamma}^{(i)}
=\operatorname{blkdiag}(S_{\Gamma\Gamma}^{(i)})\)
and \(B_{\Gamma\Gamma}^{(i)}=A_{\Gamma\Gamma}^{(i)}\).
The local Schur matrix \(S_{\Gamma\Gamma}^{(i)}\) is not required for this
choice, and
\(\operatorname{blkdiag}(A_{\Gamma\Gamma}^{(i)})\)
requires less storage than \(A_{\Gamma\Gamma}^{(i)}\).

As in the preceding case, the generalized eigenvalues are not necessarily
contained in the interval \([0,1]\). The maximum and minimum eigenvalues before
spectral selection are reported in
\Cref{tab:full-local-eig-extremes-blkAgg}. The maximum eigenvalue is
approximately \(1.504\) for all mesh sizes and coefficient contrasts. The
minimum eigenvalues correspond to the zero eigenvalue on each floating
subdomain, up to numerical roundoff.

\begin{table}[htbp]
\centering
\scriptsize
\caption{Maximum and minimum local generalized eigenvalues for
\(B_{\Gamma\Gamma}^{(i)}
=\operatorname{blkdiag}(A_{\Gamma\Gamma}^{(i)})\). Each entry is
\(\lambda_{\max}\,(\lambda_{\min})\).}
\label{tab:full-local-eig-extremes-blkAgg}

\setlength{\tabcolsep}{4pt} 

\scalebox{0.95}{%
\begin{tabular}{c|cccc}
\hline
\(H/h\)
& \(10^0\)
& \(10^1\)
& \(10^2\)
& \(10^3\)
\\
\hline
\(4\)
& \(1.5040\,(-7.97{\times}10^{-17})\)
& \(1.5041\,(-4.49{\times}10^{-16})\)
& \(1.5040\,(-4.79{\times}10^{-16})\)
& \(1.5040\,(-2.80{\times}10^{-16})\)
\\
\(8\)
& \(1.5040\,(-1.05{\times}10^{-15})\)
& \(1.5041\,(-3.24{\times}10^{-17})\)
& \(1.5040\,(-1.51{\times}10^{-15})\)
& \(1.5040\,(-1.04{\times}10^{-15})\)
\\
\(16\)
& \(1.5040\,(-2.95{\times}10^{-16})\)
& \(1.5041\,(8.58{\times}10^{-17})\)
& \(1.5040\,(-3.59{\times}10^{-16})\)
& \(1.5040\,(-2.36{\times}10^{-15})\)
\\
\(32\)
& \(1.5040\,(3.58{\times}10^{-15})\)
& \(1.5041\,(-1.11{\times}10^{-15})\)
& \(1.5040\,(2.68{\times}10^{-16})\)
& \(1.5040\,(-2.93{\times}10^{-15})\)
\\
\hline\hline
\(H/h\)
& \(10^4\)
& \(10^5\)
& \(10^6\)
&
\\
\hline
\(4\)
& \(1.5040\,(-1.96{\times}10^{-17})\)
& \(1.5040\,(-3.74{\times}10^{-16})\)
& \(1.5040\,(5.37{\times}10^{-17})\)
&
\\
\(8\)
& \(1.5040\,(-1.28{\times}10^{-16})\)
& \(1.5040\,(-7.81{\times}10^{-16})\)
& \(1.5040\,(-1.79{\times}10^{-16})\)
&
\\
\(16\)
& \(1.5040\,(-2.62{\times}10^{-15})\)
& \(1.5040\,(-1.49{\times}10^{-15})\)
& \(1.5040\,(2.78{\times}10^{-16})\)
&
\\
\(32\)
& \(1.5040\,(-8.71{\times}10^{-15})\)
& \(1.5040\,(-7.59{\times}10^{-16})\)
& \(1.5040\,(-4.17{\times}10^{-15})\)
&
\\
\hline
\end{tabular}%
}
\end{table}

For
\(B_{\Gamma\Gamma}^{(i)}
=\operatorname{blkdiag}(A_{\Gamma\Gamma}^{(i)})\),
the iteration counts and estimated condition numbers are given in
\Cref{tab:nosas-pencil-compare-blkAgg-fraction23}. The results are generally
comparable to those obtained with
\(B_{\Gamma\Gamma}^{(i)}=A_{\Gamma\Gamma}^{(i)}\); see
\Cref{tab:tab_ex_2}. Some degradation in the condition number and iteration
count is observed for certain parameter choices. However,
\(\operatorname{blkdiag}(A_{\Gamma\Gamma}^{(i)})\)
is less expensive to store and apply than the full matrix
\(A_{\Gamma\Gamma}^{(i)}\).

This reduction in local cost must be balanced against the size of the
resulting coarse space. In this experiment, approximately \(24\%\) to \(25\%\)
of the local eigenvectors are retained, which may increase the cost of the
coarse-grid solve.

\begin{table}[htbp!]
\centering
\scriptsize
\setlength{\tabcolsep}{4pt}
\renewcommand{\arraystretch}{1.15}
\scalebox{0.95}{%
\begin{tabular}{c c c c c c c c c c}
\toprule
\(H/h\) & \(N_E\) & \% ret.
& \(10^0\) & \(10^1\) & \(10^2\) & \(10^3\)
& \(10^4\) & \(10^5\) & \(10^6\) \\
\midrule
4
& 95 & 25.27
& 11 (5.54) & 17 (24.65) & 21 (51.97) & 21 (52.78)
& 21 (52.86) & 21 (52.86) & 21 (52.86) \\

8
& 187 & 24.61
& 9 (3.18) & 11 (5.19) & 11 (5.58) & 11 (5.63)
& 11 (5.63) & 11 (5.64) & 11 (5.64) \\

16
& 369 & 24.15
& 10 (4.60) & 11 (4.96) & 10 (5.18) & 10 (5.48)
& 10 (5.68) & 10 (4.97) & 9 (3.34) \\

32
& 739 & 24.12
& 16 (40.65) & 15 (29.51) & 13 (26.62) & 10 (4.95)
& 10 (4.95) & 11 (4.95) & 11 (4.95) \\
\bottomrule
\end{tabular}}
\caption{NOSAS results for the unaligned checkerboard coefficient test using
\(B_{\Gamma\Gamma}^{(i)}
=\operatorname{blkdiag}(A_{\Gamma\Gamma}^{(i)})\).}
\label{tab:nosas-pencil-compare-blkAgg-fraction23}
\end{table}
\subsection{Anisotropic jump coefficient not aligned with mesh} \label{sec:num_ex_3} 
 The experiments in \Cref{sec:num_ex_1} and \Cref{sec:num_ex_2} assumed that the coefficients were discontinuous but isotropic. Here we test the performance of the method for anisotropic jump coefficients. The same $3\times 3$ checkerboard grid from \Cref{sec:num_ex_2} is used, and the mesh is not aligned with the checkerboard.  The coefficients now take the form
    \begin{align} 
    \K =
    \begin{cases}
    \begin{bmatrix}
    10 + \rho_{2} & 0 \\
    0             & \rho_2
\end{bmatrix}     
     &\mbox{ if } E\in \mathcal{T}_h,\quad\quad\quad E\in \texttt{red}
    \\[12pt]
    \begin{bmatrix}
    \rho_{2} & 0 \\
    0             & 10+\rho_2
\end{bmatrix}    
       &\mbox{ if }E\in   \mathcal{T}_h,~~E\in \texttt{white}
    \end{cases}.
   \label{eq:num_ex_3}
    \end{align}
 So, the coefficients restricted to a red or white region are constant diagonal tensors, therefore the anisotropy ratio is  $1 + 10/\rho_2.$  We vary the parameter $\rho_2=10^{-j}$ for $j=0,1,2,\ldots,6$. For the eigenvalue threshold (see \Cref{sec:coarse_space}), we ser $\eta =2 \frac{ h}{H}.$
 
\Cref{tab:tab_ex_3} displays the results of this test case. 
We can see that as the anisotropy ratio varies, the number of iterations remains small and the condition number grows slowly. Moreover, as the ratio $\frac{H}{h}$ increases, the iterations and condition numbers remain similarly uniform.
It is apparent that the preconditioner remains robust for anisotropic jump coefficients. Moreover, the as the anisotropy ratio becomes larger, the condition number and total preconditioned Krylov iterations remain uniform. The preconditioner retains uniform convergence behavior as the mesh size is varied.
 \begin{table}[htbp]
\centering
\scriptsize
\setlength{\tabcolsep}{4pt}
\renewcommand{\arraystretch}{1.15}
\begin{tabular}{c c c c c c c c c c c}
\toprule
\(H/h\) & \(\eta\) & \(N_E\) & \% ret.
& \(10^0\) & \(10^1\) & \(10^2\) & \(10^3\)
& \(10^4\) & \(10^5\) & \(10^6\) \\
\midrule
4
& 1/2 & 882 & 21.81
& 12 (8.53) & 10 (4.71) & 11 (6.76) & 11 (6.85)
& 11 (6.86) & 11 (6.86) & 11 (6.86) \\

8
& 1/4 & 174 & 22.90
& 10 (4.89) & 9 (4.13) & 9 (4.16) & 8 (4.14)
& 8 (4.14) & 8 (4.14) & 8 (4.14) \\

16
& 1/8 & 366 & 23.95
& 8 (3.47) & 9 (3.75) & 8 (3.63) & 7 (3.67)
& 7 (3.67) & 7 (3.67) & 7 (3.67) \\

32
& 1/16 & 748 & 24.41
& 13 (4.90) & 8 (3.73) & 7 (3.45) & 7 (3.33)
& 7 (3.33) & 7 (3.33) & 7 (3.33) \\
\bottomrule
\end{tabular}
\caption{Results of the numerical experiments from \Cref{sec:num_ex_3} (jump coefficients unaligned with mesh). The computation uses a \(3\times 3\) subdomain
partition and the threshold rule \(\eta=2h/H\). The columns \(10^j\) correspond
to \(\rho_1=10^j\). Entries report Krylov iterations and, in parentheses, the
estimated condition number \({\kappa}\). Since the
number of retained local generalized eigenvectors varies slightly with \(\rho_1\),
the quantities \(N_E\) and \(\%\) retained are averaged.}
\label{tab:tab_ex_3}
\end{table}

\subsection{Effect of eigenvalue threshold} \label{sec:num_ex_4}  
We perform an analogous experiment to that in \Cref{sec:num_ex_3}, but we will study the impact $\eta$ has on the solver. Increasing $\eta$ will generally improve the number of iterations and the conditioning of the additive Schwarz operator. On the other hand, increasing $\eta$ results in a more expensive preconditioner. Keeping all eigenvalues results in a direct solver, since the coarse grid matrix coincides with the Schur complement (\Cref{eq_coarse_grid_bilinear}). In contrast, selecting only the smallest eigenvalue in each subdomain will result in the cheapest preconditioner, but its overall performance will deteriorate in general.  
\begin{table}[htbp]
\centering
\scriptsize
\setlength{\tabcolsep}{4pt}
\renewcommand{\arraystretch}{1.15}
\begin{tabular}{c c c c c c}
\toprule
\(\eta\) & \(k_i\) min/avg/max & \(N_E\) & \% ret.
& iter. & \({\kappa}\) \\
\midrule
\(1/2\)     & \(14/26.13/28\) & 4416 & 22.44 & 9  & 3.89 \\
\(1/4\)     & \(14/25.99/28\) & 4392 & 22.32 & 10 & 4.72 \\
\(1/8\)     & \(14/25.99/28\) & 4392 & 22.32 & 10 & 4.72 \\
\(1/16\)    & \(14/25.99/28\) & 4392 & 22.32 & 10 & 4.72 \\
\(11/200\)  & \(14/24.28/26\) & 4104 & 20.85 & 13 & 27.03 \\
\(1/20\)    & \(14/24.28/26\) & 4104 & 20.85 & 13 & 27.03 \\
\(9/200\)   & \(12/20.59/22\) & 3480 & 17.68 & 27 & 82.96 \\
\(1/25\)    & \(10/17.62/19\) & 2977 & 15.13 & 32 & 121.10 \\
\(7/200\)   & \(8/15.76/17\)  & 2664 & 13.54 & 37 & 146.10 \\
\(1/32\)    & \(8/13.92/15\)  & 2353 & 11.96 & 39 & 184.70 \\
\(11/400\)  & \(6/12.07/13\)  & 2040 & 10.37 & 45 & 231.10 \\
\(1/40\)    & \(6/11.21/12\)  & 1894 & 9.62  & 48 & 241.80 \\
\(1/50\)    & \(5/9.23/10\)   & 1560 & 7.93  & 56 & 358.80 \\
\(1/64\)    & \(3/7.38/8\)    & 1247 & 6.34  & 58 & 522.20 \\
\bottomrule
\end{tabular}
\caption{Experiment for \Cref{sec:num_ex_4}. The experiment fixes
a \(13\times13\) subdomain partition, \(H/h=8\), and \(\rho_1=10^6\), and varies
the spectral threshold \(\eta\). The column \(k_i\) reports the minimum,
average, and maximum number of retained local generalized eigenmodes over the
subdomains. The column \(N_E\) gives the total number of retained modes.}
\label{tab:nosas_eta_sensitivity_anisotropic_unaligned_13x13}
\end{table}
 Keeping all eigenvalues results in a direct solver, since the coarse grid matrix coincides with the Schur complement (\Cref{eq_coarse_grid_bilinear}). Conversely, selecting only the smallest eigenvalue in each subdomain will result in the cheapest preconditioner, but its overall performance will deteriorate in general. 

Table~\ref{tab:nosas_eta_sensitivity_anisotropic_unaligned_13x13} shows the
effect of the spectral cutoff \(\eta\) on the size and effectiveness of the
coarse space for the anisotropic, non-mesh-aligned coefficient test.  For a wide
range of relatively large thresholds, the retained eigenspace is essentially
unchanged: for \(\eta=1/4,1/8,\) and \(1/16\), the method retains
\(N_E=4392\) modes, the iteration count remains fixed at 10, and the estimated
condition number is \({\kappa}=4.72\).  Increasing the threshold to
\(\eta=1/2\) retains only 24 additional modes and gives only a modest
improvement in the estimated condition number.  Thus, over this range, changing
the cutoff does not significantly alter the effective coarse space.

Once the threshold is reduced below \(1/16\), however, the retained coarse space
begins to shrink and the preconditioner deteriorates.  For example, reducing
\(\eta\) from \(1/16\) to \(11/200\) decreases the number of retained modes from
\(4392\) to \(4104\), while increasing the estimated condition number from
\(4.72\) to \(27.03\).  The sharper degradation occurs near \(\eta=9/200\), where
the number of retained modes drops to \(3480\), the iteration count increases to
27, and \({\kappa}\) rises to \(82.96\).  Further reductions in
\(\eta\) lead to a much smaller coarse space and substantially larger Krylov
iteration counts.  This indicates that the relevant local generalized eigenvalues
are clustered with a clear transition region: the method is insensitive to
\(\eta\) above this transition, but becomes much less effective once the cutoff
excludes modes needed to represent the coefficient anisotropy.

\subsection{SPE10 data set} \label{sec:num_ex_5}  
For this experiment we consider a permeability field from the tenth Society of
Petroleum Engineers comparative solution project (SPE10)
\cite{christie2001tenth,SPE10}.  This data set provides highly heterogeneous
permeability fields that vary over several orders of magnitude.  Such strong
coefficient variation is a challenging test for the linear systems arising from
the IPDG discretization \eqref{eq_dg_disc2}.  The permeability field used in this
experiment is visualized in \Cref{fig:num_ex_5_0}.
\begin{figure}[htb!]
     \centering
     \subfloat[\centering SPE10 layer 70]{{\includegraphics[scale=0.4]{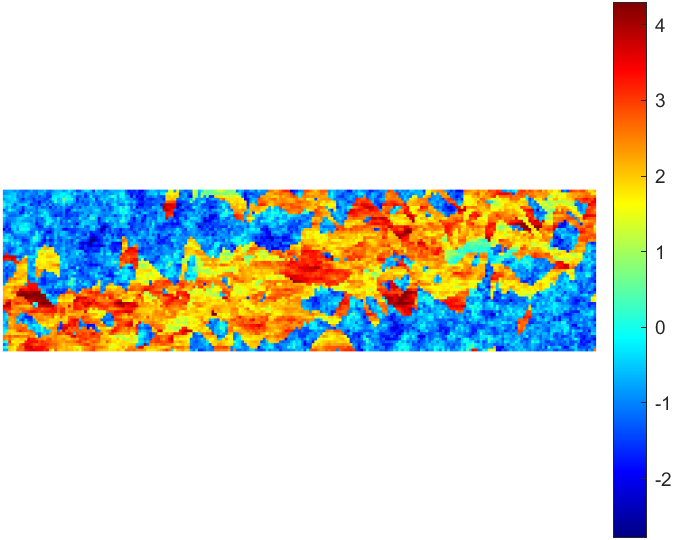} } \label{fig:num_ex_5_1}}%
     \qquad
     \subfloat[\centering SPE10 layer 85]{{\includegraphics[scale=0.4]{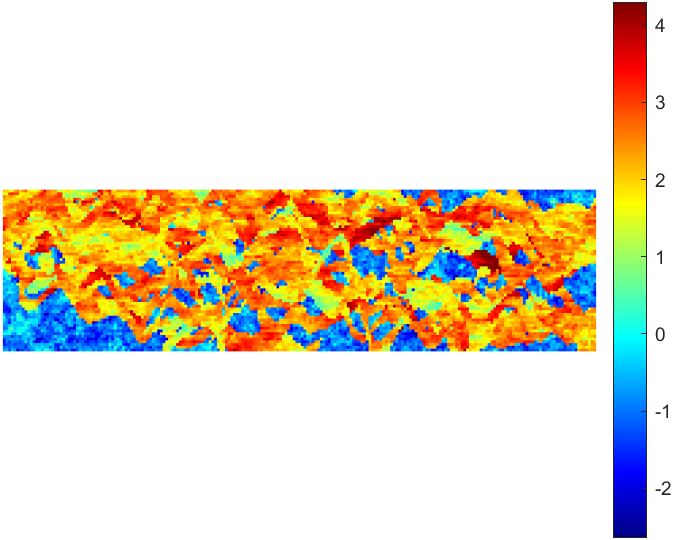} }
     \label{fig:num_ex_5_2}}%
     \caption{Various layers from the SPE10 project. The coefficients $\K$ are given in a log-log scale. It can be seen that the coefficients are highly discontinuous and vary over several orders of magnitude. 
     }%
     \label{fig:num_ex_5_0}%
 \end{figure}
We take
\[
    \Omega = [0,220]\times[0,60],
\]
matching the horizontal dimensions of the SPE10 data set.  The computational
mesh is obtained by splitting each Cartesian SPE10 cell into two triangles,
giving \(220\times60\) rectangular cells and \(26400\) triangular elements.  The
coefficient \(\K\) is taken from layer \(85\) of the SPE10 data set, using the
\(K_x\) component as a scalar isotropic permeability.  After normalization by
the minimum positive value, the resulting coefficient satisfies
\[
    1 \leq \K(x) \leq 8.803\times 10^6 .
\]
Both triangles inside a given SPE10 Cartesian cell are assigned the same
permeability value, so the coefficient jumps remain aligned with the underlying
SPE10 cell structure.

Unless otherwise stated, we use the natural SPE10 subdomain partition consisting
of \(10\times10\) Cartesian SPE cells per subdomain.  This gives a
\(22\times6\) subdomain partition, hence \(132\) subdomains, with approximately
\(200\) triangular elements per subdomain and \(H/h=10\).  We use this benchmark
to examine the sensitivity of the NOSAS coarse space to the spectral selection
parameter and to the retained coarse-space dimension.

Table~\ref{tab:nosas_spe10_fixed_fraction} reports a fixed-fraction
coarse-space sensitivity test for the SPE10 coefficient field.  In contrast to
the threshold-based selection rule, this experiment directly controls the
relative coarse-space size by retaining a prescribed fraction \(f\) of the local
generalized eigenmodes on each subdomain.  The results show a sharp tradeoff
between coarse-space size and preconditioner quality.  Retaining \(25\%\) of the
local modes gives 6 Krylov iterations and an estimated condition number
\(\widehat{\kappa}=1.72\).  Reducing the retained fraction to \(20\%\) decreases
the total number of retained modes from \(4720\) to \(3776\), but the iteration
count increases to 16 and the estimated condition number rises to \(14.75\).
Further reductions lead to progressively worse performance: at \(15\%\), \(10\%\),
and \(5\%\), the method requires 25, 29, and 45 iterations, respectively.  Thus,
for this SPE10 layer, the coarse space must retain roughly one quarter of the
local spectral modes to obtain the same robust behavior observed in the
threshold-based experiments.
\begin{table}[htbp]
\centering
\scriptsize
\setlength{\tabcolsep}{5pt}
\renewcommand{\arraystretch}{1.15}
\begin{tabular}{c c c c c c}
\toprule
\(f\) & \(k_i\) min/avg/max & \(N_E\) & \% ret.
& iter. & \({\kappa}\) \\
\midrule
\(25\%\) & \(20/35.76/40\) & 4720 & 25.28 & 6  & 1.72 \\
\(20\%\) & \(16/28.61/32\) & 3776 & 20.23 & 16 & 14.75 \\
\(15\%\) & \(12/21.45/24\) & 2832 & 15.17 & 25 & 37.94 \\
\(10\%\) & \(8/14.30/16\)  & 1888 & 10.11 & 29 & 60.54 \\
\(5\%\)  & \(4/7.15/8\)    & 944  & 5.06  & 45 & 138.40 \\
\bottomrule
\end{tabular}
\caption{Fixed-fraction spectral coarse-space sensitivity for the SPE10
coefficient test (layer 85). The experiment uses the natural \(220\times60\) SPE10
permeability field, a \(22\times6\) subdomain partition, and \(H/h=10\).
For each subdomain, the selection rule retains the prescribed fraction \(f\)
of the smallest local generalized eigenmodes. The column \(k_i\) reports
the minimum, average, and maximum number of retained modes per subdomain, while
\(N_E\) gives the total retained coarse dimension.}
\label{tab:nosas_spe10_fixed_fraction}
\end{table}

\subsection{Coarse space assessment (IPDG)}\label{sec:num_ex_6}  
As discussed in \Cref{sec:coarse_space2}, multiple coarse spaces may be used. For the DG discretization, when the coarse space is supported on the global interface $\Gamma$, the discontinuity across elements inflates the number of unknowns, particularly in comparison with the continuous Galerkin (CG) case.

In \Cref{eq:woodbury}, it was shown how to potentially reduce the complexity of the coarse grid operator (via the Sherman–Morrison–Woodbury formula \cite{horn2012matrix}). Reducing the size of the coarse grid operator is of interest because it is well-known that efficient scaling of multilevel solvers is hindered by coarse grid operators \cite{amg001,amg002}.

With respect to domain decomposition techniques, many subdomains implies a large coarse space, which can be too unwieldy (e.g., consider the extreme case of one element per subdomain). On the other end of the extreme, if there is only one subdomain with all elements, the coarse space coincides with the original discretization matrix. Therefore, a balance must be made when considering the number of subdomains.
 
In \Cref{fig:num_ex_6_1}, we consider a domain decomposition of an unstructured mesh (128 partitions). The mesh has 7074 triangular elements. The corresponding sparsity patterns of the matrices ${\bf A}_{\Gamma\Gamma}$ and ${\bf S}_{\Gamma }$ are displayed in \Cref{fig:num_ex_6_2} and \Cref{fig:num_ex_6_2_0} respectively. Although the dimensions of ${\bf A}_{\Gamma\Gamma}$ and ${\bf S}_{\Gamma }$ are the same, we note that ${\bf A}_{\Gamma\Gamma}$ is significantly more sparse: 431066 nonzeros compared to 33636 nonzeros. This is over a complete order of magnitude difference. Although the dimension of the coarse space does matter, a better assessment of computational cost of direct solvers for sparse matrices is the number of nonzero entries \cite{davis2006direct}.
\begin{figure}[htb!]
    \centering
    \subfloat[\centering mesh partition]{{\includegraphics[clip,scale=0.35]{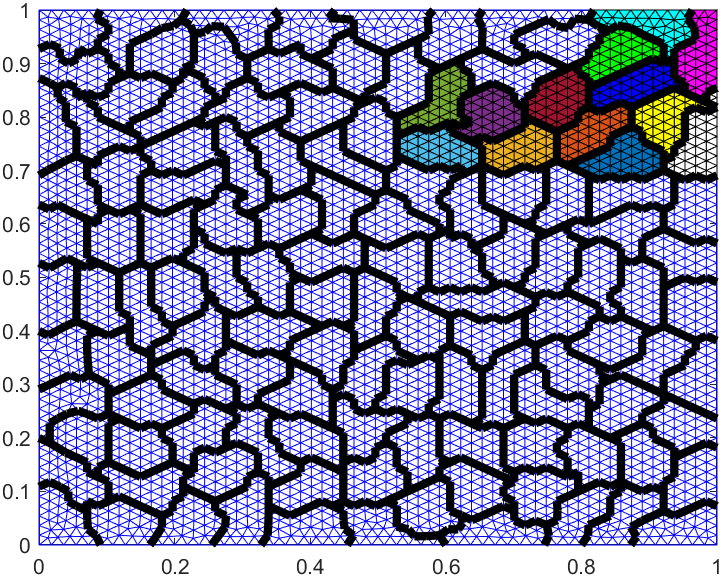} } \label{fig:num_ex_6_1}}%
    \subfloat[\centering sparsity of ${\bf A}_{\Gamma\Gamma}$]{{\includegraphics[clip,scale=0.39]{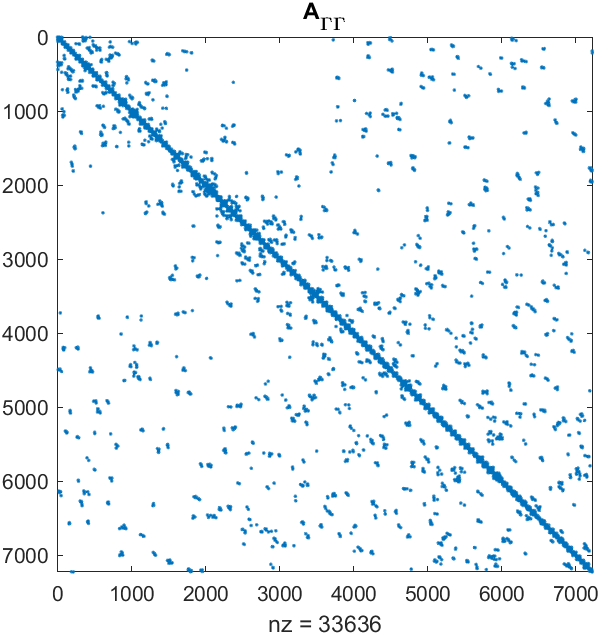} }
    \label{fig:num_ex_6_2}}
    %
    \subfloat[\centering sparsity of ${\bf S}_{\Gamma }$]{{\includegraphics[clip,scale=0.39]{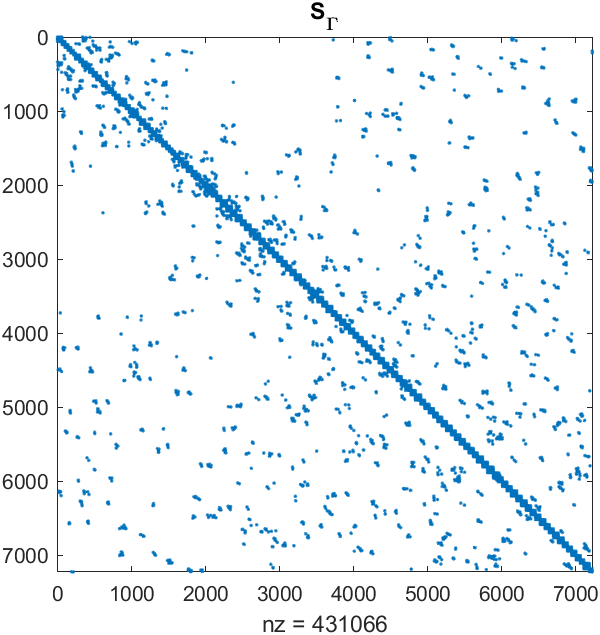} }
    \label{fig:num_ex_6_2_0}}
    \caption{
Comparison of the sparsity of ${\bf A}_{\Gamma\Gamma}$ and ${\bf S}_{\Gamma }$ for an unstructured mesh partitioning (128 partitions).
    }%
    \label{fig:num_ex_6_0}%
\end{figure} 

\begin{figure}[htb!]
    \centering
    \subfloat[\centering mesh partition]{{\includegraphics[clip,scale=0.35]{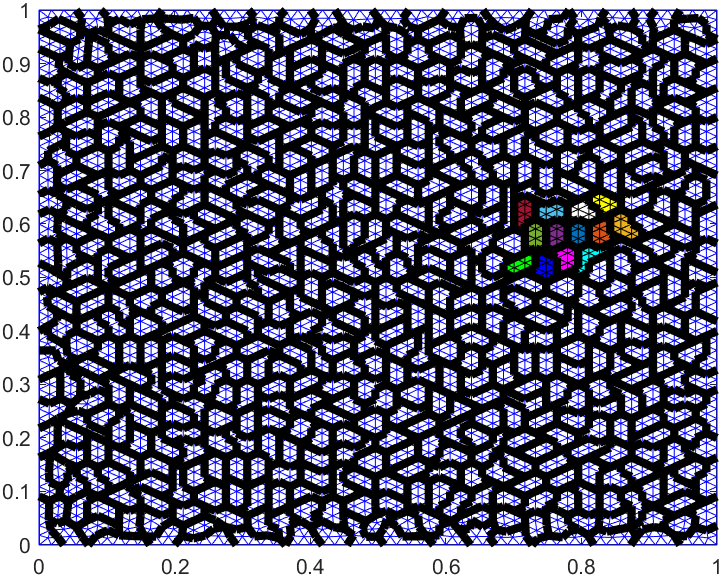} } \label{fig:num_ex_6_3}}%
    \subfloat[\centering sparsity of ${\bf A}_{\Gamma\Gamma}$]{{\includegraphics[clip,scale=0.39]{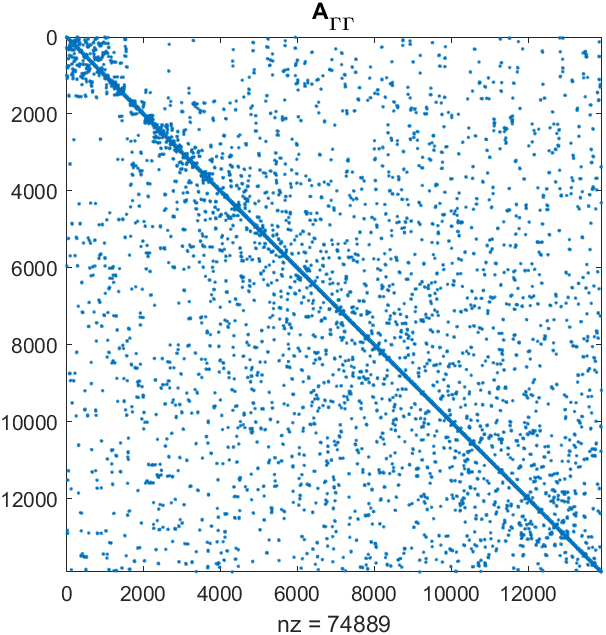} }
    \label{fig:num_ex_6_4}}%
    \subfloat[\centering sparsity of ${\bf S}_{\Gamma }$]{{\includegraphics[clip,scale=0.39]{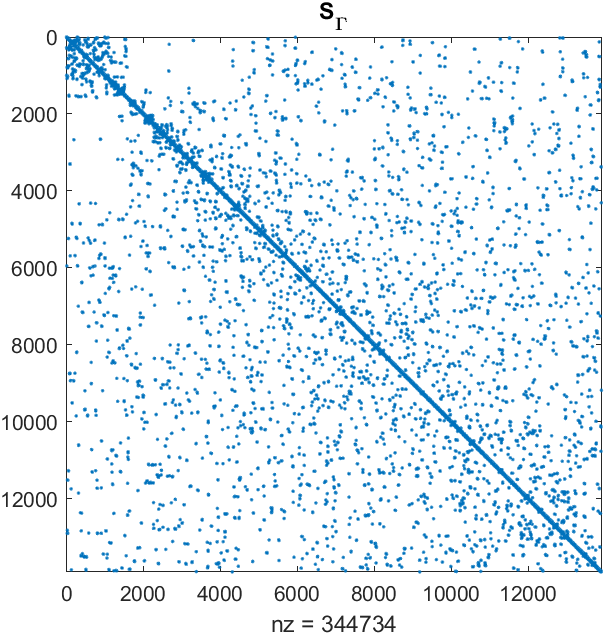} }}        
    \caption{ 
Comparison of the sparsity of ${\bf A}_{\Gamma\Gamma}$ and ${\bf S}_{\Gamma }$ for an unstructured mesh partitioning (600 partitions).    
    } 
    \label{fig:num_ex_6_00}%
\end{figure} 
\Cref{fig:num_ex_6_00} Provides a similar comparison, but this time using a larger number of subdomains (600 partitions). Here we can still observe a significant reduction in the number of nonzero entries in ${\bf A}_{\Gamma\Gamma}$.

We also remark that the storage can be further reduced by leveraging the observation that ${\bf A}_{\Gamma\Gamma}$ is symmetric. If ${\bf A}_{\Gamma\Gamma}$ can be factored efficiently (e.g., Cholesky decomposition \cite{davis2006direct}), then the computational cost of the coarse space is reduced to a $N_E\times N_E$ matrix $ 
{\bf C}^{-1}{\bf D}^{-1}
-
{\bf U}^T{\bf A}_{\Gamma\Gamma}^{-1}{\bf U} 
$ (see \Cref{eq:woodbury}); where $N_E$ is the total number of eigenvalues (dimension of the eigenspace). From \Cref{sec:matrix_coarse}, one can deduce that the matrices ${\bf C}$ and ${\bf D}$ are block-diagonal, so inversion is straightforward.
 
\subsection{Coarse space assessment (CG)}\label{sec:num_ex_7}  
We continue the assessment from \Cref{sec:num_ex_6}, but comparing the efficiency of a continuous Galerkin coarse space.  For finite element discretizations, the notion of auxiliary space preconditioning is considered a state-of-the-art technique \cite{xu1996auxiliary}. These methods allow for the reuse of effective preconditioners for classical discretizations which have been studied intensively (such as continuous Galerkin).
 
For reproducibility purposes, we take the domain to be $\Omega=[0,1]^2$ and consider a structured mesh. The domain is then partitioned into $N$ uniform squares with diameter $H$. Each of these squares are further subdivided into smaller squares of diameter $h$. Each of the squares of diameter $h$ are divided into two equal area right triangles.
 
For each subdomain, Approximately 20\% of the smallest eigenpairs are retained for this example. The parameter $\K$ is the anisotropic version introduce in \Cref{sec:num_ex_3}.  As outlined in \Cref{sec:coarse_space2}, we examine two different coarse spaces, an inherited and noninherited coarse grid operator. The inherited coarse grid operator is formed by the Galerkin triple product: $\tilde{{\bf S}}_\Gamma = R_{\text{DG}}^{\text{CG}} {\bf S}_\Gamma (R_{\text{DG}}^{\text{CG}})^T$. The noninherited coarse grid operator is simply the CG discretization of the model problem \Cref{eq:model}. This coarse grid is noninherited in the sense that it does not necessarily retain information from the prolongation operator $\R_0^T$.

	Generally, to ensure efficient preconditioning, auxiliary space methods require pre-smoothing or post-smoothing \cite{bramble2019multigrid,xu1996auxiliary}. After an application of the additive Schwarz preconditioner (with a continuous Galerkin course grid operator), we apply a single post-smoothing step to the DG system consisting of a block Jacobi (block size is ${d+k \choose d}$, where $d$ is the dimension and $k$ is the polynomial degree). From various computational experiments with a CG coarse grid, pre/post-smoothing was found to be necessary for convergence. Additional smoothing steps reduces the iteration counts, but increases the computational cost. 

\Cref{tab:tab_ex_7_1} contains the results for the Galerkin triple product coarse space. If the ratio $\frac{H}{h}$ is fixed, we see a constant number of iterations are needed to reach convergence, and a condition number that increases very slowly.
\begin{table}[htb!]
\centering 
\scalebox{0.9}{
\begin{tabular}{c c c c}
 \hline 
 $(H,h)$ & iterations &  $\kappa$ &  $\dim {\bf S}_{\Gamma}$ \\ 
 \hline 
$(2^{-2},2^{-3})$ &54& 7.2520e+01& 33 \\
$(2^{-3},2^{-4})$&42& 6.5144e+01& 161\\
$(2^{-4},2^{-5})$&44& 6.7636e+01& 705\\
$(2^{-5},2^{-6})$&41& 6.9401e+01& 2945\\
$(2^{-2},2^{-4})$&36& 4.9754e+01& 81   \\    
$(2^{-3},2^{-5})$&36& 4.5032e+01& 385  \\    
$(2^{-4},2^{-6})$&32& 4.6166e+01& 1665 \\    
$(2^{-5},2^{-7})$&32& 4.6954e+01& 6913 \\    
 \hline 
 \end{tabular}}   
\caption{Results of the numerical experiments from \Cref{sec:num_ex_7}. The estimated condition number of the Schwarz operator is given by $\kappa$.  The column denoted by $\dim {\bf S}_{\Gamma}$ reports the size of the CG coarse grid operator (Galerkin triple product coarse grid). 
}
\label{tab:tab_ex_7_1}
\end{table}
 
\begin{table}[htb!]
\centering
\scalebox{0.9}{
\begin{tabular}{c c c c}
 \hline 
 $(H,h)$ & iterations &  $\kappa$ &  $\dim {\bf S}_{\Gamma}$ \\ 
 \hline 
$(2^{-2},2^{-3})$&45&  6.3377e+01& 33   \\        
$(2^{-3},2^{-4})$&51&  1.0855e+02& 161\\ 
$(2^{-4},2^{-5})$&60&  4.2334e+02& 705\\ 
$(2^{-5},2^{-6})$&93&  1.7125e+03& 2945 \\   
$(2^{-2},2^{-4})$&37 & 5.0978e+01& 81\\   
$(2^{-3},2^{-5})$&46 & 1.9357e+02& 385\\  
$(2^{-4},2^{-6})$&69 & 8.5629e+02& 1665\\ 
$(2^{-5},2^{-7})$&131& 3.7272e+03 &6913\\
 \hline 
 \end{tabular}}
\caption{Results of the numerical experiments from \Cref{sec:num_ex_7}. The estimated condition number of the Schwarz operator is given by $\kappa$.  The column denoted by $\dim {\bf S}_{\Gamma}$ reports the size of the CG coarse grid operator (rediscretization coarse grid). 
}
\label{tab:tab_ex_7_2}
\end{table}
\Cref{tab:tab_ex_7_2} contains the results for the rediscretized coarse space. If the ratio $\frac{H}{h}$ is fixed, the number of iterations required to reach convergence increases as the subdomain size increases, and the condition number increases more rapidly. The noninherited space is easier to construct, but does not retain the important information from the local generalized eigenvalue problems.

	The inherited CG coarse space performs much better than the noninherited CG coarse space. Moreover, the coarse space size is the same for both cases. Here, the CG interface coarse space is over 4 times smaller than the DG interface coarse space.

\subsection{Nonsymmetric IPDG schemes}\label{sec:num_ex_8} 
The focus of this section is to explore the efficacy of the proposed additive Schwarz preconditioner to the nonsymmetric IPDG scheme (as described in \Cref{sec:non_sym}). In the previous sections we focused on the SIPG scheme. The nonsymmetric IPDG schemes posses attractive features (e.g., improved stability or fewer terms in the bilinear form) and are used in many applications \cite{bastian2014fully,bastian2012algebraic,dolejvsi2008semi,fabien2022numerical,li2015high,liu2009modeling,riviere2008discontinuous,sun2005symmetric}.  
 On the other hand, the loss of symmetry introduces additional challenges for linear solvers.

	We repeat the numerical experiments outlined in \Cref{sec:num_ex_1,sec:num_ex_2,sec:num_ex_3,sec:num_ex_4}. The user-defined parameter is fixed as $\sigma=15$ unless stated otherwise. As the discretization matrix is no longer symmetric, we use the BICG Krylov subspace method as an outer solver. For brevity, we do not include results using the GMRES method, as similar outcomes were observed. 
\subsubsection{NIPG and IIPG for isotropic jump coefficients aligned with mesh}
\Cref{tab:tab_ex_8_1} contains the results of repeating the experiment from \Cref{sec:num_ex_1} for the NIPG and IIPG schemes. In the first column, the notation ``(NIPG,IIPG)'' stands for the number of preconditioned Krylov steps required to reduce the relative residual smaller than $10^{-6}$. The preconditioner works well for both the NIPG and IIPG schemes, independent of the mesh size and jump coefficient magnitude. 
\begin{table} [htb!]
\centering
\small
\renewcommand{\arraystretch}{1.2}
\scalebox{0.9}{
\begin{tabular}{l| l l l l l l l}
\hline 
$\rho_1$                      & $10^0$ & $10^1$ & $10^2$ & $10^3$ & $10^4$ & $10^5$ & $10^6$
 \\ \hline 
$\frac{H}{h}=4$  &   &   &    &   &   &  & 
  \\ 
(NIPG,IIPG) & (19,25) & (20,19) & (23,22) & (22,21) & (21,21) & (22,22) & (22,22)
 \\ \hline 
$\frac{H}{h}=8$  &   &   &    &   &   &  & 
  \\
(NIPG,IIPG) & (17,18) & (21,35) & (23,20) & (23,23) & (23,24) & (23,23) & (22,22)
 \\ 
\hline 
$\frac{H}{h}=16$  &   &   &    &   &   &  & 
  \\
(NIPG,IIPG) & (26,35) & (22,23) & (22,26) & (23,22) & (22,22) & (23,22) & (22,21)\\ 
\hline 
\end{tabular}  }
\caption{NIPG and IIPG schemes for \Cref{sec:num_ex_1} (isotropic jump coefficients aligned with mesh).}
\label{tab:tab_ex_8_1}
\end{table}
\subsubsection{NIPG and IIPG for isotropic jump coefficients unaligned with mesh}
We repeat the done experiment in \Cref{sec:num_ex_2} for the NIPG and IIPG schemes. In \Cref{tab:tab_ex_8_2} we report the results. Although the jump coefficients are unaligned with the mesh, the preconditioner performs well for the nonsymmetric schemes. This robust behavior persists through changes in the mesh and subdomain size. Compared to the SIPG results from \Cref{sec:num_ex_2}, we see that the NIPG and IIPG schemes require slightly more total iterations.
\begin{table} [htb!]
\centering
\small
\renewcommand{\arraystretch}{1.2}
\scalebox{0.9}{
\begin{tabular}{l| l l l l l l l}
\hline 
$\rho_1$                      & $10^0$ & $10^1$ & $10^2$ & $10^3$ & $10^4$ & $10^5$ & $10^6$
 \\ \hline 
$\frac{H}{h}=4$  &   &   &    &   &   &  & 
  \\ 
(NIPG,IIPG) & (18,19) & (20,21) & (25,23) & (23,21) & (22,23) & (22,23) & (24,24)
 \\ \hline 
$\frac{H}{h}=8$  &   &   &    &   &   &  & 
  \\
(NIPG,IIPG) & (17,18) & (19,19) & (24,21) & (22,19) & (21,20) & (20,21) & (21,21)
 \\ 
\hline 
$\frac{H}{h}=16$  &   &   &    &   &   &  & 
  \\
(NIPG,IIPG) & (26,28) & (20,19) & (21,23) & (19,18) & (19,19) & (21,19) & (21,19)
\\ 
\hline 
\end{tabular}  }
\caption{NIPG and IIPG schemes for \Cref{sec:num_ex_2} (isotropic jump coefficients unaligned with mesh). }
\label{tab:tab_ex_8_2}
\end{table}
 \subsubsection{NIPG and IIPG for anisotropic jump coefficients unaligned with mesh}
Next we revisit the experiment conducted in \Cref{sec:num_ex_3}, this time for the NIPG and IIPG schemes. We gather from \Cref{tab:tab_ex_8_3} that the preconditioner for NIPG and IIPG remains robust even for anisotropic jump coefficients unaligned with the mesh. Compared to the SIPG results from \Cref{sec:num_ex_2}, we see that the NIPG and IIPG schemes require slightly more total iterations.
\begin{table} [htb!]
\centering
\small
\renewcommand{\arraystretch}{1.2}
\scalebox{0.9}{
\begin{tabular}{l| l l l l l l l}
\hline 
$\rho_1$                      & $10^0$ & $10^1$ & $10^2$ & $10^3$ & $10^4$ & $10^5$ & $10^6$
 \\ \hline 
$\frac{H}{h}=4$  &   &   &    &   &   &  & 
  \\ 
(NIPG,IIPG) &(27,27) & (23,29) & (21,20) & (19,19) & (18,18) & (19,24) & (18,20)
 \\ \hline 
$\frac{H}{h}=8$  &   &   &    &   &   &  & 
  \\
(NIPG,IIPG) & (23,23) & (21,22) & (19,19) & (21,19) & (17,18) & (18,18) & (17,18)
 \\ 
\hline 
$\frac{H}{h}=16$  &   &   &    &   &   &  & 
  \\
(NIPG,IIPG) & (23,23) & (22,22) & (19,21) & (21,20) & (18,18) & (20,22) & (18,20)
\\ 
\hline 
\end{tabular}  }
\caption{NIPG and IIPG schemes for \Cref{sec:num_ex_3} (anisotropic jump coefficients unaligned with mesh). 
}
\label{tab:tab_ex_8_3}
\end{table}

\subsubsection{NIPG and IIPG eigenvalue threshold effect}
In \Cref{sec:num_ex_4} it was demonstrated how the performance of the preconditioner for the SIPG method changed depending on the eigenvalue threshold. We examine the impact that this eigenvalue threshold parameter has on the NIPG and IIPG schemes.
 \begin{table} [htb!]
\centering
\small
\renewcommand{\arraystretch}{1.2}
\scalebox{0.9}{
\begin{tabular}{l| l l l l }
\hline 
$\frac{H}{h}=8$  &   &   &    &   
  \\\hline 
$\%$ ret                     & $22\%$ & $18\%$ & $10\%$ & $6\%$  \\ \hline 
(NIPG,IIPG)  & 
(13,11)   &  
(20,22)   & 
(35,30)   &  
(112,104)   
\\
\hline 
\end{tabular}  }
\caption{NIPG and IIPG schemes for \Cref{sec:num_ex_4} ($\%$ ret is the approximate number of eigenpairs retained for each subdomain). 
}
\label{tab:tab_ex_8_4}
\end{table}
Similar to \Cref{sec:num_ex_4}, from \Cref{tab:tab_ex_8_4} we gather that increasing $\mu$ results in fewer preconditioned iterations. Conversely, decreasing $\mu$ has the effect of reducing the number of eigenvalues per subdomain, but resulting in more iterations to reach tolerance. In practice the choice of $\mu$ will depend on the coefficient $\K$, but the preconditioner is typically cheaper to apply for smaller values of $\mu$. However, the total iteration counts may suffer.

 \subsubsection{NIPG dependence on penalty parameter}
One of the interesting features about the NIPG scheme is that it can be thought of as parameter free, since any $\sigma>0$ will provide a stable method \cite{riviere2008discontinuous}. However, the choice of penalty parameter $\sigma$ does effect the conditioning of the linear system. We examine the effect $\sigma$ has on the total iteration count for the NIPG scheme in \Cref{fig:nipg_penalty}. The test case from \Cref{sec:num_ex_3} is used, with $\frac{H}{h}=8$, $\rho_2=\{2^{-1},2^{-3},2^{-5}\}$, and the penalty parameter $\sigma=2^j$ for $j=-2,-1,0,1,2,\ldots,7$.
\begin{figure}[htb!]
\centering
	\includegraphics[scale=0.5]{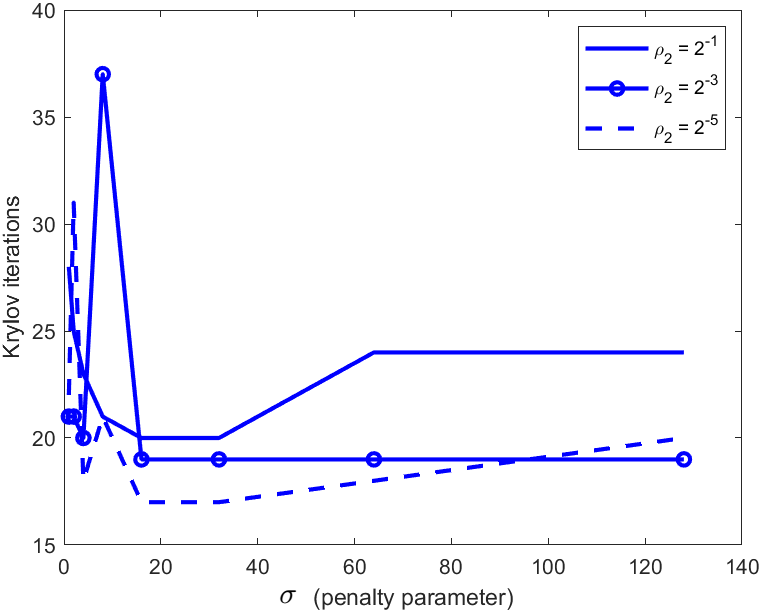}
\caption{Preconditioned Krylov iterations versus NIPG penalty parameter.}
\label{fig:nipg_penalty}	
\end{figure}
From \Cref{fig:nipg_penalty} we can deduce that the penalty parameter has a slight impact on the total iteration count.  We remark that the complete penalty parameter takes into account a harmonic averaging of $\K$ (see \eqref{eq:dg_penalty}), as well as $\sigma$. SIPG and IIPG are much more sensitive for smaller values of $\sigma$.

	In summary, the additive Schwarz method applied to the NIPG and IIPG schemes give robust and scalable results; similar to the SIPG case. The NIPG scheme is less sensitive for smaller penalty parameters, but the preconditioner performs well for both the NIPG and IIPG schemes in the case of unaligned anisotropic jump coefficients. According to the numerical experiments conducted in \Cref{sec:num_ex_1,sec:num_ex_2,sec:num_ex_3,sec:num_ex_4}, the SIPG method has fewer total iterations than the NIPG or IIPG schemes. The eigenvalue threshold $\mu$ is still important to balance for NIPG and IIPG, since it influences how efficient the preconditioner is.
  
\section{Conclusions}
	In this paper we designed and analyzed a nonoverlapping additive Schwarz preconditioner for interior penalty discontinuous Galerkin discretizations of anisotropic elliptic problems. The preconditioned method coupled with a Krylov subspace iteration is shown to be independent of the highly discontinuous (and anisotropic) jump coefficients as well as the subdomain size. A key aspect of this work is the utilization of the so-called spectral coarse spaces, which consider local generalized eigenvalue problems in each subdomain \cite{spillane2014abstract}.

	To increase efficacy, various auxiliary spaces are considered to reduce the size of the coarse grid operator. It is determined that certain continuous Galerkin auxiliary spaces can be suitable coarse grid candidates, but care must be taken to ensure the low-rank discrete energy harmonic extension in each subdomain is retained approximately.

    Several options for the local generalized eigenvalue problems were introduced and theoretically analyzed. It is found that the variant which uses the block diagonal of the Schur complement is very robust, but expensive, since the Schur complement matrix is needed to be assembled and stored.  The variant which uses the block diagonal of ${\bm A}_{\Gamma\Gamma}^{(i)}$ allows for more efficient computations, while simultaneously preserving strong solver performance.
    
	 We also demonstrated how to modify the additive Schwarz preconditioner such that it is applicable to the nonsymmetric IPDG schemes. Several numerical experiments verified the theory and validated the robustness of the preconditioner. Future work includes extending these ideas to three and multilevel DD methods, as well as overlapping methods. 

\bibliographystyle{siamplain}
\bibliography{references}
\end{document}